\documentclass{cmslatex}

\usepackage{amssymb,amsmath}
\usepackage{graphicx}
\usepackage{float}
\usepackage{color}
\usepackage{psfrag,epsfig}
\usepackage{calc}
\newtheorem{remark}{Remark}[section]

\newcommand{\N}  {{\mathbb N}}
\newcommand{\Z}  {{\mathbb Z}}
\newcommand{\R}  {{\mathbb R}}
\newcommand{\cP}  {{\mathcal P}}
\newcommand{\cL}  {{\mathcal L}}
\newcommand{\cC}  {{\mathcal C}}
\newcommand{\cF}{\mathcal F}

\newcommand{\abs}[1]{\lvert#1\rvert}
\newcommand{\norm}[1]{\lVert#1\rVert}
\newcommand{\dsp}{\displaystyle}
\renewcommand\theequation{\thesection.\arabic{equation}}            
\definecolor{mygreen}{rgb}{0.11,0.7,0.22}
\definecolor{mred}{rgb}{1,0.,0.}
\definecolor{mcyan}{rgb}{1.,0.5,0.}
\definecolor{mgreen}{rgb}{0.,0.5,0.}
\definecolor{mblue}{rgb}{0.,0.,1.}
\definecolor{mbluebis}{rgb}{0.,0.75,1.}

\newenvironment{enum_spe}{ 
\begin{list}{} {
	\setlength{\labelwidth}{3.7cm}\setlength{\leftmargin}{\labelwidth+\labelsep}}}{\end{list}}
	        
\title{Transparent boundary conditions for locally perturbed infinite hexagonal  periodic media\thanks{The
    authors are partially supported by the French ANR fundings under
    the project MicroWave NT09\_460489.}}

\author{C.Besse\footnotemark[2]
\and J. Coatl\'{e}ven\footnotemark[3]
\and S. Fliss\footnotemark[3]
\and I. Lacroix-Violet\footnotemark[2]
\and K. Ramdani\footnotemark[4]}

\begin{document}

\maketitle

\renewcommand{\thefootnote}{\fnsymbol{footnote}}

\footnotetext[2]{Laboratoire Paul Painlev\'e, 
  Universit\'e Lille Nord de France, CNRS UMR 8524, INRIA SIMPAF Team,
  Universit\'e Lille 1 Sciences et Technologies,
  Cit\'e Scientifique,
  59655 Villeneuve d'Ascq Cedex, France. 
({\tt Christophe.Besse@math.univ-lille1.fr}, {\tt Ingrid.Violet@math.univ-lille1.fr}).}

\footnotetext[3]{Laboratoire POems (UMR 7231 CNRS-INRIA-ENSTA) 
32, Boulevard Victor, 75739 PARIS Cedex 15
({\tt sonia.fliss@ensta-paristech.fr}, {\tt julien.coatleven@inria.fr}).}

\footnotetext[4]{
Universit\'e de Lorraine, IECN, UMR 7502, 
Vandoeuvre-l\`es-Nancy, F-54506, France.

CNRS, IECN, UMR 7502, 
Vandoeuvre-l\`es-Nancy, F-54506, France.

Inria, Villers-l\`es-Nancy, F-54600, France.({\tt karim.ramdani@inria.fr})
}

\renewcommand{\thefootnote}{\arabic{footnote}}



\begin{abstract}
	In this paper, we propose a strategy to determine the Dirichlet-to-Neumann (DtN) operator for infinite, lossy and locally perturbed hexagonal periodic media. We obtain a factorization of this operator involving two non local operators. The first one is a DtN type operator and corresponds to a half-space problem.  The second one is a Dirichlet-to-Dirichlet (DtD) type operator related to the symmetry properties of the problem. The half-space DtN operator is characterized via Floquet-Bloch transform, a family of elementary strip problems and a family of stationary Riccati equations. The DtD operator is the solution of an affine operator valued equation which can be reformulated as a non standard integral equation.
\end{abstract}



\pagestyle{myheadings}
\thispagestyle{plain}
\markboth{Besse, Coatleven, Fliss, Lacroix-Violet,
  Ramdani}{Transparent boundary conditions for locally perturbed
  infinite hexagonal  periodic media}


\section{Introduction and problem setting}\label{sec:intro}

Periodic media appear in many physical and engineering applications related to wave type systems. In solid state theory, it is well-known that the existence of insulators and conducting materials can be explained by the periodic properties of the crystal. More precisely, within the one-electron model of solids, electrons move under a Hamiltonian with periodic potential yielding the existence of band gaps (see for example Reed and Simon \cite[p.\,312]{ReeSim78}). In optics, many devices used in microtechnology and nanotechnology involve materials with such electromagnetic properties, known as photonic crystals. For a general introduction to the physics of photonic crystals, we refer the interested reader to the monographs by Joannopoulos {\it et al.} \cite{Joetal95}, Johnson and Joannopoulos \cite{JoJo02}, Sakoda \cite{Sak04} or the review paper by Busch {\it et al.} \cite{BusFreLin07}. Concerning mathematical aspects related to photonic crystals, see for instance the book of Kuchment \cite{Ku01}. About twenty years ago, the analysis of elastic wave propagation in periodically structured media lead to the concept of phononic crystals (see for instance the book by Maldovan and Thomas \cite{MalTho08}). \\\\
Although they concern materials involving quite different scales and types of waves (acoustic, electromagnetic, elastic), photonic, phononic or real crystals (in solids) share a lot. The main common feature, inherited from the periodic structure, is the appearance of band-gaps (forbidden bands), i.e. the strong attenuation of a certain range of frequencies (at least in some directions). This phenomenon is due to the fact that an incident wave on the crystal is multiply scattered by the periodic structure, leading to possibly destructive interference (depending on the characteristics of the crystal and the wave frequency).\\\\
Most applications involving periodic structures use these band gaps to control wave propagation properties of materials (see \cite{BusFreLin07} for an extended list of references). More precisely, the aim can be to prohibit propagation of waves 	\cite{RobIII98}, allow propagation of some frequencies and/or some directions like in optical filters \cite{QiuMulSwi03}, localize waves by creating a defect in the periodic structure \cite{LevSte84,Sak04}, create optical nanocavities \cite{AkaAsaSon05} or to guide waves with bends \cite{TalGouBou02}, Y--junctions \cite{BosMidKra02} and T--junctions \cite{FanJohJoa01}.\\
In the above applications, numerical simulation plays a crucial role in the design of photonic, phononic and real crystals. In particular, the problem of determining which physical and geometrical properties of a perfect or imperfect crystal yield band gaps at specified frequencies, guiding waves along some desired paths or trapping them in a given cavity is computationally challenging. The complexity of designing such devices is increased by the strong scattering effects involved, due to the high contrast of the composing materials on the one hand, and the geometric scales (wavelength and size of the periodic cell) on the other hand. Therefore, several numerical methods and tools have been developed to simulate  wave propagation in {\it infinite periodic} media. A first class of methods covers problems where the periodicity can be handled via homogenization techniques \cite{AlCoVa99,BoGuZo10}. Indeed, when the wavelength is much larger than the period, the periodic medium behaves asymptotically like a homogeneous one. The unboundedness of the medium can then be treated analytically using integral equations, Dirichlet-to-Neumann (DtN) operators or Perfectly Matched Layers techniques. On the contrary, a second class of methods keeps the periodicity but considers only  
\begin{itemize}
	\item media which are finite (i.e. constituted of a finite number of periodic cells embedded in an infinite homogeneous domains) \cite{EhrHanZhe09,EhZh08,HuLu08,YuLu06,YuLu07};
	\item media which can be reduced to finite domains. The case of imperfect\footnote{A local defect is generally created by changing locally the physical/geometrical properties of the crystal.} periodic media has been more intensively investigated \cite{FigKle97,FigKle98,Gau07}. In these works, the problems considered lead to exponentially decaying solutions, which allows the truncation of the infinite periodic media.
\end{itemize}
More recently, DtN operators have been derived for infinite two-dimensional periodic media containing local defects in \cite{TheseSonia,FlissAPNUM}. The main assumptions concern the directions of periodicity (orthogonal), the corresponding periodicity lengths (commensurate) and the presence of a dissipative term (arbitrarily small). 

\begin{figure}[ht]
  \centering
	\includegraphics[width=.4\textwidth]{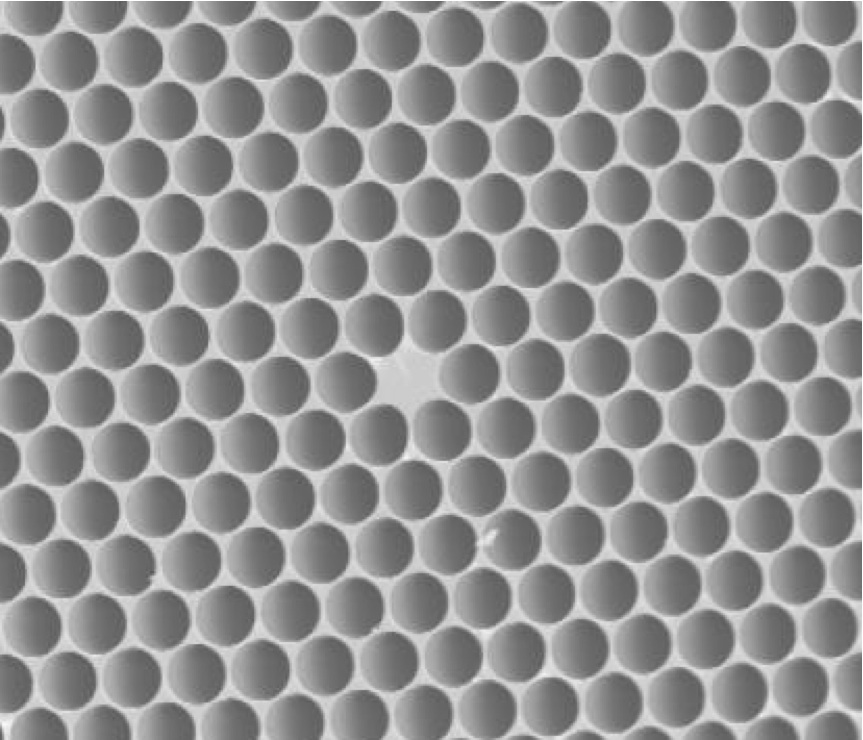}
	\hspace{0.1\textwidth}
  \includegraphics[width=.4\textwidth]{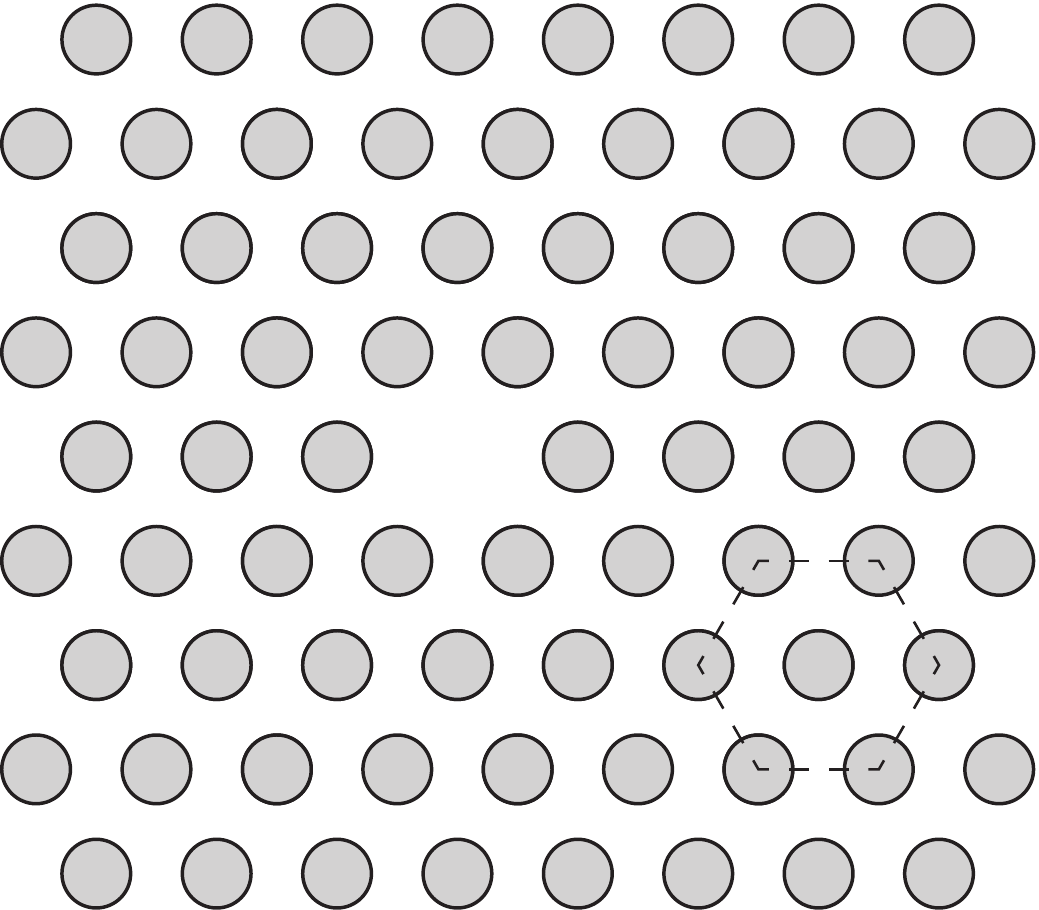}
  \caption{A locally perturbed photonic crystal with triangular lattice (cross section of a photonic crystal optical fiber (left) and representation of the corresponding hexagonal periodicity cell (right)).}
  \label{fig:photonic_crystal}
\end{figure}

~\\~\noindent In this work, we are interested in the analysis of hexagonal periodic media containing a local defect (see Figure~\ref{fig:photonic_crystal}). We call hexagonal periodic medium (also known as hexagonal lattice) a two dimensional domain where
\begin{itemize}
	\item the angle between the periodicity directions is $\pi/3$;
	\item the periodicity cell has hexagonal symmetry (see Definition \ref{defrotation}).
\end{itemize} 
Note that like in \cite{TheseSonia,FlissAPNUM}, one could consider the media as a periodic one with two orthogonal directions of periodicity. However, the corresponding periods would not be commensurate in this case. 

\smallskip
\noindent Hexagonal lattices appear in quantum mechanics
\cite{ELi09,Ker92,ReeSim78}, 
phononics \cite{Le-Rao11,PhaFle06,SpaRuzGon09} and photonics \cite{DosByrBot06,EdeHel09,GauMnaNew08}. Although they concern different applications and involve different types of waves, the corresponding problems are quite similar from the mathematical point of view. Essentially, as far as numerical simulation is concerned, the main issue is to determine a transparent boundary condition to reduce the problem initially set on an unbounded domain to a problem set on a bounded one containing the defect cell. The differential operators involved are the ones describing the underlying physics of the problem under consideration, namely
\begin{itemize}
	\item {\bf Quantum mechanics:} in a (classical) crystal, the mathematical formulation of the problem leads to the {\bf Schr\"odinger operator} 
	\begin{equation}\label{Aschrodinger}	
		Au:=-\Delta u+(V+ip)u
	\end{equation}	
	where $V(x)$ denotes the potential in the lattice and $p$ is the Laplace variable.
	\item {\bf Phononics:} the operator involved is the {\bf elasticity system} (see \cite{AmmKanLee09,AviGriMia05,Le-Rao11,PhaFle06,RupEvgMau07,TanYanTam07})
	\begin{equation}\label{Aelasticite}
	A{\mathbf u}:=-\text{div\,}{\boldsymbol\sigma}({\mathbf u})+\omega^2\rho {\mathbf u},	
	\end{equation}
where ${\mathbf u}$ denotes the displacement and 
${\boldsymbol\sigma}({\mathbf u})={\mathbf C}:{\boldsymbol\varepsilon}({\mathbf u})$ the stress tensor, in which  ${\boldsymbol\varepsilon}({\mathbf u})$ stands for the strain tensor and ${\mathbf C}$ the 4$th$-order elasticity tensor ($\rho$ and $\omega$ respectively denote the mass density and the time frequency).
	\item {\bf Photonics:} in this case, electromagnetic propagation is described by the vector {\bf Maxwell's equations}. For two-dimensional photonic crystals, these equations reduce to the following scalar equations respectively in the cases of TE (Transverse Electric) and TM (Transverse Magnetic) polarizations:
	\begin{equation}\label{AEMG}	
	\left\{
	\begin{array}{ll}
		\text{TE case: } & \dsp Au := \Delta u + \omega^2n^2 u\\
		\text{TM case: } &\dsp  Au := -\text{div\,}\left(\frac{1}{n^2}\nabla u\right) + \omega^2 u
	\end{array} 
	\right.
	\end{equation}
	where $n(x)$ denotes the index of refraction and $\omega$ the wavenumber.
\end{itemize}
Note that all the above operators involve an elliptic principal part. Moreover, the time dependence has been eliminated via a Fourier (or Laplace) transform, leading to the appearance of the real parameter $\omega$ (or the complex parameter $p$). Perfect periodic media are described via the operators \eqref{Aschrodinger}, \eqref{Aelasticite} or \eqref{AEMG} --depending on the considered application-- with periodic coefficients (potential, elasticity tensor and mass density or index, respectively for classical, phononic and photonic crystals). 
\\\\
The introduction of a defect in the periodic structure is taken into account by adding a bounded obstacle or locally perturbing the coefficients. Typically, scattering by an impurity in a classical crystal (i.e. when the lattice contains a different atom at one of the lattice points), we have 
$$V=V_{\text{per}}+V_0$$ 
where $V_{\text{per}}$ is periodic and $V_0$ is a short range potential describing the local perturbation (see \cite[p.\,312]{ReeSim78}).

\begin{figure}[ht]
  \centering
\psfrag{a}[c][][1.0]{$\Omega^i$}
\psfrag{b}[c][][1.0]{$\Sigma^i$}
\psfrag{e}[l][][1.0]{$\Sigma^0$}
\psfrag{c}[c][][1.0]{$x$}
\psfrag{d}[r][][1.0]{$y$}
\psfrag{f}[c][][1.0]{$O$}
\psfrag{g}[c][][1.0]{${\bf e}_1$}
\psfrag{h}[c][][1.0]{${\bf e}_2$}
\psfrag{i}[c][][1.0]{$d$}
  \includegraphics[width=.7\textwidth]{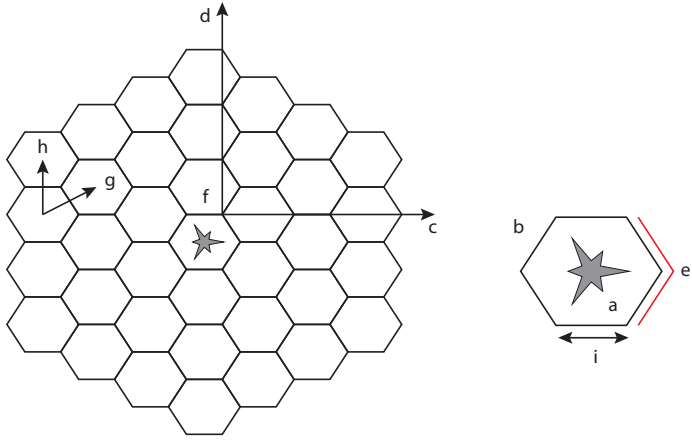}
  \caption{The hexagonal periodic medium with defect.}
  \label{fig:geometry}
\end{figure}

~\\~\noindent In this paper, we will restrict our study, for the sake of simplicity, to the case of an infinite photonic crystal $\Omega=\R^2$ containing a localized defect in the hexagonal cell $\Omega^i$ (whose side length is denoted by $d$). We denote by  $\Sigma^i=\partial\Omega^i$ the boundary of this cell and by $\Omega^e=\Omega\setminus\overline{\Omega^i}$  its exterior. From the mathematical point of view, the model problem we consider throughout the paper is a dissipative Helmholtz equation
\begin{equation}\label{Helmholtz}
\Delta u +\rho u= f, \qquad\hbox{in }\Omega,
\end{equation}
where the following assumptions will always be supposed to hold true 
\begin{itemize}
	\item {\bf (A1)} $\rho$ is a local perturbation of a hexagonal periodic function $\rho_{\text{per}}$. More precisely:
	\[
		\rho = \rho_{\text{per}} +\rho_0
	\]
	where 
	\begin{itemize}
		\item for all ${\bf x}=(x,y)\in \Omega$ and all $ (p,q)\in\Z^2,\; \rho_{\text{per}}\left({\bf x}+p{\bf e}_1+q{\bf e}_2\right)=\rho_{\text{per}}({\bf x})$ with ${\bf e}_1=(3d/2,\sqrt{3}d/2)$ and ${\bf e}_2=(0,\sqrt{3}d)$ the two directions of periodicity of the media (see Figure~\ref{fig:geometry});
		\item $\rho_{\text{per}}$ and $\rho_0$ have hexagonal symmetry (see Definition \ref{defsymfonc});
		\item Supp$(\rho_0)\subset \Omega^i$.
	\end{itemize}
	\item {\bf (A2)} $\rho$ satisfies the dissipation property 
	\begin{equation}\label{HypDissipation}
		|{\rm Im}\, \rho({\bf x})|\geq \rho_b>0, \qquad\forall {\bf x}\in \Omega.
	\end{equation}
	\item {\bf (A3)} The source $f$ is compactly supported in $\Omega^i$ and has hexagonal symmetry.
\end{itemize}
Before going further, let us make some comments about the above assumptions. 
In {\bf (A2)}, condition \eqref{HypDissipation} guarantees the
existence and uniqueness of finite energy solutions of
\eqref{Helmholtz} (i.e. solutions in $H^1(\Omega)$) in the infinite
domain $\Omega$ (as it can be easily checked by using Lax-Milgram
Lemma). When no dissipation is assumed (i.e. when $\rho$ is
real-valued), such existence and uniqueness issues will not be
discussed in this paper, nor will be the derivation of a limiting
absorption principle. Let us emphasize that the latter remains, to our
knowledge, an open question (see \cite[Remark 2]{FlissAPNUM} and \cite{JolLiFli06} for a similar discussion). \\ 
\\\\
The main goal of this paper is to propose a method to solve \eqref{Helmholtz} in the infinite domain $\Omega$ under assumptions {\bf (A1)-(A2)-(A3)}. 
The main steps of our approach  are presented in a formal way in the next Section. 

\section{Formal presentation of the method}
For the reader's convenience, we describe in this section the broad outlines of our approach in order to help the reader getting an overview of the proposed strategy. Willfully, we decided to focus on the description of the mains steps of the method and to skip in this part all the technical details related to the functional framework.
\\\\
Our approach is adapted from the one used in \cite{TheseSonia} for the case of square lattices. The key idea is to reduce problem \eqref{Helmholtz}, which is set in the unbounded domain $\Omega$, to a boundary value problem set in the cell $\Omega^i$ containing the defect. To achieve this, we need to derive a suitable transparent boundary condition on $\Sigma^i$ associated with a DtN operator $\Lambda$.
More precisely, we note that the restriction $u^i:=u\left|_{\Omega^i}\right.\in H^1(\Omega^i)$ solves the interior boundary value problem
\[
\left\{
\begin{array}{ll}
\Delta u^i + \rho u^i = f,&\qquad\hbox{in }\Omega^i, \\[5pt]
\displaystyle\frac{\partial u^i}{\partial \nu^i} + \Lambda u^i=0,&\qquad\hbox{on } \Sigma^i,
\end{array}
\right.
\]
where $\nu^i$ is the outgoing unit normal to $\Omega^i$ and $\Lambda$ denotes the DtN operator defined by
\begin{equation}\label{eqLambda}
\Lambda \phi=-\displaystyle\left.\frac{\partial u^e(\phi)}{\partial \nu^i}\right|_{\Sigma^i}
\end{equation}
in which $u^e(\phi) \in H^1(\Omega^e)$ is the unique solution of the exterior problem
\begin{equation}
\label{Pbext}
\left\{
\begin{array}{ll}
\Delta u^e(\phi) + \rho u^e(\phi)= 0,& \qquad\hbox{in }\Omega^e, \\[5pt]
u^e(\phi)=\phi,& \qquad \hbox{on } \Sigma^i.
\end{array}
\right.
\end{equation}
The core of the paper is thus devoted to the computation of this DtN operator $\Lambda$. First of all, let us emphasize that due to the symmetry properties of the original problem ($2\pi/3$ rotational invariance, see Assumptions {(\bf A1)} and {(\bf A3)}), it suffices to compute this DtN operator $\Lambda$ for Dirichlet data  $\phi$ on $\Sigma^i$ with hexagonal symmetry (See section \ref{Pbdefault} for more details). In this Section, we restrict our analysis from now on to such symmetric data.
\\\\
In order to describe the different steps of our method, we need to introduce some additional notation.
\begin{figure}[ht]
  \centering
\psfrag{b}[l][][1.0]{$\Sigma^H$}
\psfrag{c}[l][][1.0]{$\left.u^H(\phi)\right|_{\Sigma^H}=\phi$}
\psfrag{d}[l][][1.0]{$u^H(\phi)$}
\psfrag{e}[l][][1.0]{$\Omega^H$}
\psfrag{f}[l][][1.0]{$\Sigma^H$}
\psfrag{g}[l][][1.0]{$\left.\dsp\Lambda^H(\phi)\right|_{\Sigma^H}=\left.\displaystyle\frac{\partial u^H(\phi)}{\partial\nu^H}\right|_{\Sigma^H}$}
\psfrag{h}[l][][1.0]{$u^H(\phi)$}
\psfrag{i}[l][][1.0]{$\Omega^H$}
\psfrag{j}[l][][1.0]{$\Lambda^H$}
  \includegraphics[width=\textwidth]{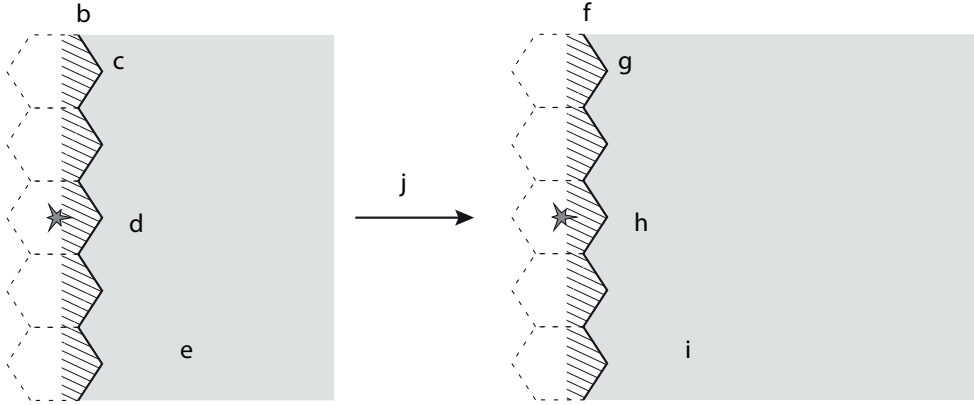}
  \caption{The half-space DtN operator $\Lambda^H$}
  \label{notationssec3}
\end{figure}
\begin{figure}[ht]
  \centering
\psfrag{a}[l][][1.0]{$\Sigma^i$}
\psfrag{c}[c][][1.0]{$\Omega^i$}
\psfrag{d}[c][][1.0]{$\Omega^i$}
\psfrag{b}[l][][1.0]{$\Sigma^H$}
\psfrag{f}[l][][1.0]{$\phi$}
\psfrag{e}[l][][1.0]{$u^e(\phi)$}
\psfrag{g}[l][][1.0]{$\left.u^e(\phi)\right|_{\Sigma^H}$}
\psfrag{k}[c][][1.0]{$D_{2\pi/3}$}
  \includegraphics[width=\textwidth]{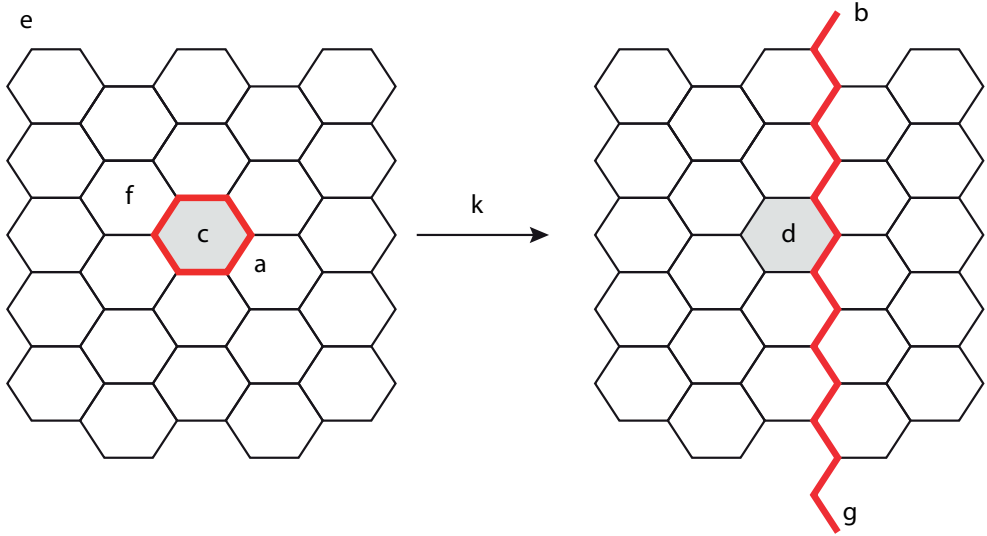}
  \caption{The DtD operator $D_{2\pi/3}$}
  \label{fig:d2pis3}
\end{figure}
Let $\Sigma^H$ be the boundary depicted in Figure~\ref{notationssec3} and let $\Omega^H$ be the half-space to right of $\Sigma^H$. Given a Dirichlet data $\phi$ on $\Sigma^{H}$, we define the half-space DtN operator  $\Lambda^H$ by setting (see Figure~\ref{notationssec3})
\begin{equation}\label{eqLambdaH_new}
\Lambda^H \phi =\left. \frac{\partial u^H(\phi)}{\partial\nu^H}\right|_{\Sigma^H}
\end{equation}
where $\nu^H$ is the outgoing unit normal to $\Omega^H$ and $u^H(\phi)$ is the unique solution in $H^1(\Delta,\Omega^H)$ of the half-space problem
\begin{equation}
\label{PbDemiespsec3_new}
(\mathcal{P}^H)\quad\left\{
\begin{array}{ll}
\Delta u^H(\phi) + \rho u^H(\phi) =0, &\qquad\hbox{in } \Omega^H, \\[5pt]
u^H(\phi)= \phi, & \qquad\hbox{on }\Sigma^H.
\end{array}
\right.
\end{equation}
The second important ingredient needed is the DtD operator $D_{2\pi/3}$ defined from the boundary $\Sigma^i$ to $\Sigma^H$ (see Figure~\ref{fig:d2pis3}) by the formula:
\begin{equation}\label{eqD2pi3_new}
D_{2\pi/3}\phi = \left.u^e(\phi)\right|_{\Sigma^H}
\end{equation}
where $u^e(\phi)$ is the unique solution of the exterior problem \eqref{Pbext} for a given Dirichlet data $\phi$ on $\Sigma^{i}$. 
\subsubsection*{\bf Step 1 : Factorization of the DtN operator $\Lambda$}
With the above notation, it is clear that given a Dirichlet data  $\phi$ on $\Sigma^i$, the functions $u^e(\phi)$ and $u^H(D_{2\pi/3}\phi)$ are both solutions of the half-space homogeneous Helmholtz problem 
$$
\Delta U + \rho U =0,
$$
with the same Dirichlet condition on $\Sigma^H$, namely $D_{2\pi/3} \phi.$ The uniqueness of the solution of this problem implies that 
$$u^e(\phi)\left|_{\Omega^H}\right.=u^H(D_{2\pi/3}\phi),$$ 
and in particular, the corresponding normal derivatives coincide on the part of the boundary $\Sigma^i\cap \Sigma^H$  where they are both defined, yielding
\begin{equation}\label{eqLambdafact_new}
\left.\Lambda \phi\right|_{\Sigma^i\cap \Sigma^H} =\left. \Lambda^H \left(D_{2\pi/3}\phi\right)\right|_{\Sigma^i\cap \Sigma^H},
\end{equation}
the remaining part of $\Lambda \phi$ on $\Sigma^i$ being recovered using hexagonal symmetry. The above relation constitutes the starting point of our strategy : the DtN operator $\Lambda$ can be computed via the factorization formula \eqref{eqLambdafact_new}.
We provide in Theorem \ref{THsec3_2} a more precise statement of this {\em factorization result}, paying a particular attention to the functional framework. Thus, our problem is now reduced to the computation of the half-space DtN operator $\Lambda^{H}$ and the DtD operator $D_{2\pi/3}$ for well prepared data ($i.e.$ having $2\pi/3$ rotational invariance).
\subsubsection*{\bf Step 2 : Characterization of the DtN operator $\Lambda^H$}
In order to compute the half-space DtN operator $\Lambda^H$, the first key ingredient is the (partial) Floquet-Bloch transform in the vertical direction (see \S.\ref{sub:solution_of_the_half-space_problem}). More precisely, applying it to the Helmholtz half-space problem \eqref{PbDemiespsec3_new}, we will see that it suffices to consider the case of $k-$quasiperiodic Dirichlet data $\phi$ on $\Sigma^H$, that is 
$$
\phi(y+qL)=\phi(y)e^{iqkL}, \qquad \forall  y\in \R, \ \forall q\in \Z, 
$$
where $L = \sqrt{3}d$ denotes the period in the vertical direction and $k\in (-\pi/L,\pi/L)$. From now on we will thus restrict our analysis to $k-$quasiperiodic Dirichlet data.
\\\\
The second important tool we need is the so-called {\em propagation operator} $\cP_k$, defined as follows (see \S\ref{ssub:the_propagation_operator}). For any $k-$quasiperiodic Dirichlet data $\phi$ on $\Sigma^H$, $\cP_k\,\phi$ is nothing but the trace of the solution $u^H(\phi)$ of \eqref{PbDemiespsec3_new} on the trasnalted vertical boundary $\widetilde{\Sigma}^{H}=\Sigma^H+{\bf e}_{1}$ (see Assumption {\bf (A1)}). The main advantage of this operator is that it allows us to determine the solution $u^H(\phi)$ of the half-space problem in any cell from the knowledge of $u^H(\phi)$ on a reference cell.
\begin{figure}[htbp]
  \centering
\psfrag{a}[c][][1.0]{$\mathcal{C}_{00}$}
\psfrag{b}[c][][1.0]{$\mathcal{C}_{10}$}
\psfrag{c}[c][][1.0]{$\mathcal{C}_{01}$}
\psfrag{e}[c][][1.0]{$\mathcal{C}_{20}$}
\psfrag{f}[c][][1.0]{$\mathcal{C}_{02}$}
\psfrag{d}[c][][1.0]{$\Omega_0$}
\includegraphics[width=.3\textwidth]{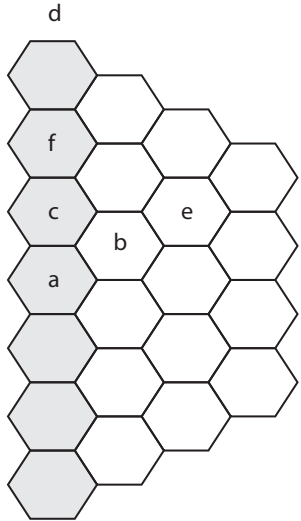}
\caption{Description of the half-space $\Omega^H$ and its periodicity cells}
\label{phiperiode}
\end{figure}
More precisely, with the notation given in Figure \ref{phiperiode}, we can prove thanks to a uniqueness argument that
$$
\left.u^H(\phi)\right|_{\mathcal{C}_{pq}} = e^{iqkL}\left.u^H((\cP_k)^p\phi)\right|_{\mathcal{C}_{00}}.
$$
According to the above relation, solving the half-space problem amounts to determining the propagator $\cP_k$ and the values of $u^H(\phi)$ but only in the reference cell ${\mathcal{C}_{00}}$. For the latter, by linearity, this can be achieved by solving two elementary cell problems set in ${\mathcal{C}_{00}}$ (see \S.\ref{ssub:elementary_cell_problems}). Regarding the determination of the propagator operator $\cP_k$, we proceed as follows (see \S.\ref{ssub:the_ricatti_equation_for_the_determination_of_the_propagator}). Writing the matching of the normal derivatives of the solution $u^H(\phi)$ across a suitable part of the interface $\widetilde{\Sigma}^H$, we will show that $\cP_k$ solves a stationary Riccati equation.
\subsubsection*{\bf Step 3 : Characterization of DtD operator $D_{2\pi/3}$}
While the determination of the half-space DtN operator $\Lambda^{H}$ described in Step 2 essentially uses the periodicity properties of the medium, the determination of the DtD operator $D_{2\pi/3}$ crucially uses its hexagonal symmetry. More precisely, we show in Theorem \ref{THsec3_3} that $D_{2\pi/3}$ solves an affine operator-valued equation (see equation \eqref{PbD} in Theorem \ref{THsec3_3}) that is well-posed. In order to handle this equation from a practical point of view, we use once again Floquet-Bloch variables instead of the physical ones. Doing so, we reduce this affine equation to the solution of a set of non standard constrained integral equations (see equations \eqref{PbD_FB}).

\begin{remark}
For the sake of clarity, we have restricted our analysis in this paper to the model problem \eqref{Helmholtz}, under assumptions {\bf (A1)-(A2)-(A3)}. From the physical point of view, this problem describes the radiation in a photonic crystal of a source localized (typically a real source) in one cell of the periodic medium. We would like to emphasize that our approach can be generalized to tackle other types of problems: extended defects ($i.e.$ covering more than one periodic cell), non penetrable defects or scattering problems by a local inhomogeneity/obstacle in the crystal. Non symmetric configurations can also be considered but involves more intrincate integral equations (see \cite{TheseSonia} for non symmetric square lattice).
\end{remark}

The precise description of the three steps briefly outlined above involves many technical details related to the functional framework that needs to be carefully settled for the DtN and DtD operators. In Section \ref{Pbdefault}, we collect some notational and mathematical background related to geometries with hexagonal symmetry. Section \ref{Sect_FACT} provides the factorization result (Theorem \ref{THsec3_2}), which plays a key role in our method for computing the DtN operator $\Lambda$. The next two sections are devoted to the analysis of the two non local operators $\Lambda^H$ and $D_{2\pi/3}$ involved in the factorization. More precisely, Section \ref{SPbdemiesp} deals with the half-space DtN operator $\Lambda^H$, while the DtD operator $D_{2\pi/3}$ is studied in Section \ref{sec:opextension}. An algorithm will be given in conclusion to summarize this construction.


\section{Hexagonal symmetry and related results}
\label{Pbdefault}

\subsection{Definitions and functional framework}
\begin{definition}
\label{defrotation}
A domain ${\cal O}$ of $\R^2$ has hexagonal symmetry if there exists a rotation of angle $2\pi/3$, denoted $\Theta_{2\pi/3}$ for which ${\cal O}$ is invariant.
\end{definition}
\noindent In the above definition, the center of the rotation is not specified to simplify the notation. In the following, this center will always be the center of gravity of ${\cal O}$.
\begin{remark}
If a domain ${\cal O}$ of $\R^2$ has hexagonal symmetry, its boundary $\partial {\cal O}$ also has hexagonal symmetry.
\end{remark}
\begin{definition}
\label{defsymfonc}
Let ${\cal O}$ be an open set with hexagonal symmetry and let $g$ be a real or complex valued function defined on ${\cal O}$. Then, $g$ has hexagonal symmetry if $$g=g \circ \Theta_{2\pi/3}.$$
\end{definition}
\noindent Let ${\cal O}$ be an open domain with hexagonal symmetry (typically ${\cal O} =\Omega^e$, $\Omega^i$ and $\partial {\cal O}=\Sigma^i$). For $s\ge 0$, we denote by $H^s_{2\pi/3}({\cal O})$ the closed subspace of $H^s({\cal O})$ defined by 
\begin{equation}\label{eq:def_Hs}
	H^s_{2\pi/3}({\cal O})=\{ v \in H^s({\cal O}),~v=v \circ \Theta_{2\pi/3} \}.
\end{equation}
One can easily check that the restriction to $H^1_{2\pi/3}({\cal O})$ of the trace operator $\gamma_0$ defined by
$$\forall u \in H^1({\cal O}),\quad\gamma_0 u=u\left|_{\partial {\cal O}}\right.\in H^{1/2}(\partial{\cal O}),$$
defines a continuous operator from $H^1_{2\pi/3}({\cal O})$ onto $H^{1/2}_{2\pi/3}(\partial{\cal O}).$ 
\\\\
Finally, let us define the appropriate functional space for the normal trace of a function with hexagonal symmetry. To achieve this, we introduce the spaces 
$$H^1(\Delta,{\cal O})= \{u\in H^1({\cal O})\mid \Delta u \in L^2({\cal O}) \}$$ 
and 
$$H^1_{2\pi/3}(\Delta,{\cal O})= \{u\in H^1_{2\pi/3}({\cal O})\mid \Delta u \in L^2_{2\pi/3}({\cal O})\}.$$
The next result expresses in term of functional spaces the commutativity of the Laplace operator with the rotation $\Theta_{2\pi/3}.$ 
\begin{theorem}
\label{THsec3_1}
Let ${\cal O}$ be an open set of $\R^2$ with hexagonal symmetry. The Laplace operator commutes with any unitary transform 
and maps $H^1_{2\pi/3}(\Delta,{\cal O})$ onto $L^2_{2\pi/3}({\cal O}).$
\end{theorem}
\noindent We can now extend Definition~\eqref{eq:def_Hs} to the space $H^{-1/2}(\partial {\cal O})$ where ${\cal O}$ is with hexagonal symmetry (typically ${\cal O} = \Omega^e$ or $\Omega^i$ and $\partial{\cal O}=\Sigma^i$)
\begin{definition}
\label{defenstrace}
Let ${\cal O}$ be an open set of $\R^2$ with hexagonal symmetry. 
We define the closed subspace $H^{-1/2}_{2\pi/3}(\partial {\cal O})$ of $H^{-1/2}(\partial {\cal O})$ by
$$H^{-1/2}_{2\pi/3}(\partial {\cal O})=\left\{\gamma_1 u=\left.\frac{\partial u}{\partial \nu}\right|_{\partial {\cal O}}, u \in H^1_{2\pi/3}(\Delta,{\cal O})\right\}$$
where $\nu$ is the outgoing unit normal to ${\cal O}$.\\\\
Obviously, $\gamma_1$ is a continuous application from $H^1_{2\pi/3}(\Delta,{\cal O})$ onto $H^{-1/2}_{2\pi/3}(\partial {\cal O}).$
\end{definition}

\subsection{Restriction and extension operators}
\label{subsect_RestrictionOperators}
Let $R$ be the restriction operator defined by 
\[
	\begin{array}{rcl}
	R\::\:L^2(\Sigma^i)&\rightarrow& L^2(\Sigma^0)\\[5pt]
	\phi&\mapsto& \phi|_{\Sigma^0}
	\end{array}
\]
where $\Sigma^0$ is the right part of $\Sigma^i$ (see Figure~\ref{fig:geometry}). One can easily check that $R$ defines an isomorphism from the subspace $L^2_{2\pi/3}(\Sigma^i):=\{ v \in L^2(\Sigma^i),~v=v \circ \Theta_{2\pi/3} \}$ onto $L^2(\Sigma^0)$ that we shall denote by $R_{2\pi/3}$. Its inverse $E_{2\pi/3}$ is an extension operator which can be given explicitly thanks to the rotation $\Theta_{2\pi/3}$:
\[
	\forall \phi \in L^2(\Sigma^0),\quad 
	\begin{array}{|l} \left.E_{2\pi/3}\phi\right|_{\Sigma^0} = \phi\\[3pt]
	\left.E_{2\pi/3}\phi\right|_{\Theta_{2\pi/3}\Sigma^0} =\phi\circ  \Theta_{-2\pi/3}\ \\[3pt]
	\left.E_{2\pi/3}\phi\right|_{\Theta_{2\pi/3}^2\Sigma^0} = \phi\circ  \Theta_{-2\pi/3}^2
\end{array}
\]
We define the following space
$$
\begin{array}{rcl}
	\displaystyle H^{1/2}_{2\pi/3}(\Sigma^0) &:=& \displaystyle\left\{R_{2\pi/3}\phi,\; \phi \in H^{1/2}_{2\pi/3}(\Sigma^i)\right\}, \\[5pt]
	&=& \displaystyle \left\{\phi\in H^{1/2}(\Sigma^0),\; E_{2\pi/3}\phi \in H^{1/2}_{2\pi/3}(\Sigma^i)\right\}.
\end{array}
$$
This space is nothing but the space of functions in $H^{1/2}(\Sigma^0)$ which are periodic.\\\\
We now explain how to extend the restriction operator $R_{2\pi/3}$ to $H^{-1/2}_{2\pi/3}(\Sigma^i)$ (see Definition \ref{defenstrace}). This can be done by duality, noticing that we have
\[
	\forall\phi\in L^2_{2\pi/3}(\Sigma^i),\;\forall\psi\in L^2(\Sigma^0),\quad \left( R_{2\pi/3}\phi,\psi\right)_{\Sigma^0} = \frac{1}{3}\left( \phi,E_{2\pi/3}\psi\right)_{\Sigma^i},
\]
where $(\cdot,\cdot)_{\Sigma^0}$ (resp. $(\cdot,\cdot)_{\Sigma^i}$) is the scalar product in $L^2(\Sigma^0)$ (resp. $L^2(\Sigma^i)$). This last relation suggests an extension of $R_{2\pi/3}$ to $H^{-1/2}_{2\pi/3}(\Sigma^i)$ by
\[
	\forall\phi\in H^{-1/2}_{2\pi/3}(\Sigma^i),\;\forall\psi\in H^{1/2}_{2\pi/3}(\Sigma^0),\quad \langle R_{2\pi/3}\phi,\psi\rangle_{\Sigma^0} = \frac{1}{3}\langle \phi,E_{2\pi/3}\psi\rangle_{\Sigma^i},
\]
where $
\langle\cdot,\cdot\rangle_{\Sigma^0}$ (resp. $\langle\cdot,\cdot\rangle_{\Sigma^i}$) is the duality product between the two spaces $[H^{1/2}_{2\pi/3}(\Sigma^0)]^\prime$ and $H^{1/2}_{2\pi/3}(\Sigma^0)$ (resp. $H^{-1/2}_{2\pi/3}(\Sigma^i)$ and $H^{1/2}_{2\pi/3}(\Sigma^i)$ ).
We introduce the closed subspace of $\left[H^{1/2}_{2\pi/3}(\Sigma^0)\right]^\prime$:
\[
	H^{-1/2}_{2\pi/3}(\Sigma^0) := R_{2\pi/3}\left(H^{-1/2}_{2\pi/3}(\Sigma^i)\right),
\]
and conclude that $R_{2\pi/3}$ is a linear continuous map from $H^{-1/2}_{2\pi/3}(\Sigma^i)$ onto $H^{-1/2}_{2\pi/3}(\Sigma^0)$.\\\\
Analogously, $E_{2\pi/3}$ can be extended to a linear continuous mapping from $H^{-1/2}_{2\pi/3}(\Sigma^0)$ onto $H^{-1/2}_{2\pi/3}(\Sigma^i)$ using:
\[
	\forall\psi\in H^{-1/2}_{2\pi/3}(\Sigma^0),\;\forall\phi\in H^{1/2}_{2\pi/3}(\Sigma^i),\quad \langle E_{2\pi/3}\psi,\phi\rangle_{\Sigma^i} = 3\langle \psi,R_{2\pi/3}\phi\rangle_{\Sigma^0}.
\]

\subsection{Symmetry properties for the exterior problem}
Using the definitions and properties of media and functions with hexagonal symmetries and considering assumptions (A1) and (A2), one can show the following theorem:
\begin{theorem}
\label{corsec3_1}
If $\phi \in H^{1/2}_{2\pi/3}(\Sigma^i)$, then the unique solution $u^e(\phi)$ of \eqref{Pbext} belongs to $H^{1}_{2\pi/3}(\Delta,\Omega^e)$ and $\Lambda \phi \in H^{-1/2}_{2\pi/3}(\Sigma^i)$, where $\Lambda$ is the DtN operator defined in \eqref{eqLambda}.
\end{theorem}
\begin{proof}
It suffices to prove that $u^e(\Theta_{2\pi/3}\phi)$ is also a solution of \eqref{Pbext} and then conclude using a uniqueness argument. The last implication follows immediately from the definition of the normal trace operator.
\end{proof}

\noindent A consequence of Theorem \ref{corsec3_1} is that $\Lambda$ maps continuously $H^1_{2\pi/3}(\Sigma^i)$ onto $H^{-1/2}_{2\pi/3}(\Sigma^i)$, leading to the natural definition of the DtN operator
\begin{equation}\label{Lambda2pi3}
\Lambda_{2\pi/3}:=\left.\Lambda\right|_{H^{1/2}_{2\pi/3}(\Sigma^i)} \in \cL (H^{1/2}_{2\pi/3}(\Sigma^i),H^{-1/2}_{2\pi/3}(\Sigma^i)).
\end{equation}
Finally, taking into account the assumptions {\bf (A1)-(A3)} concerning the hexagonal symmetry of the local perturbation $\rho_0$ and of the source term $f$, and using uniqueness argument, we deduce that
\begin{theorem}\label{th:Interior_pb}
	Let $u$ be the unique solution of \eqref{Helmholtz}. Then the restriction $u^i=u\big|_{\Omega^i}$ of $u$ is the unique solution of the interior boundary value problem
	\[
	\left\{
	\begin{array}{ll}
	\Delta u^i + \rho u^i = f,&\qquad\hbox{in }\Omega^i, \\[5pt]
	\displaystyle\frac{\partial u^i}{\partial \nu^i} + \Lambda_{2\pi/3} u^i=0,&\qquad\hbox{on } \Sigma^i.
	\end{array}
	\right.
	\]
\end{theorem}
\noindent The rest of the paper is devoted to the determination of the DtN operator $\Lambda_{2\pi/3}$.
\section{Factorization of the DtN operator}\label{Sect_FACT}

\noindent First of all let us recall some useful notation. Let $\Sigma^H$ be the boundary depicted in Figure~\ref{notationssec3} and let $\Omega^H$ be the half-space to right of $\Sigma^H$. \\\\
Let $\Lambda^H\in {\cal L}(H^{1/2}(\Sigma^H),H^{-1/2}(\Sigma^H))$ be the half-space DtN operator (see Figure~\ref{notationssec3})
\begin{equation}\label{eqLambdaH}
\Lambda^H \phi =\left. \frac{\partial u^H(\phi)}{\partial\nu^H}\right|_{\Sigma^H}
\end{equation}
where $\nu^H$ is the exterior normal to $\Omega^H$ and $u^H(\phi)$ is the unique solution in $H^1(\Delta,\Omega^H)$ of the half-space problem
\begin{equation}
\label{PbDemiespsec3}
(\mathcal{P}^H)\quad\left\{
\begin{array}{ll}
\Delta u^H(\phi) + \rho u^H(\phi) =0, &\qquad\hbox{in } \Omega^H, \\[5pt]
u^H(\phi)= \phi, & \qquad\hbox{on }\Sigma^H.
\end{array}
\right.
\end{equation}
Let $D_{2\pi/3}\in {\cal L}( H^{1/2}(\Sigma^i),H^{1/2}(\Sigma^H))$ be the DtD operator defined by (see Figure~\ref{fig:d2pis3}):
\begin{equation}\label{eqD2pi3}
D_{2\pi/3}\phi = \left.u^e(\phi)\right|_{\Sigma^H}
\end{equation}
where $u^e(\phi)$ is the unique solution of \eqref{Pbext}.
\\\\
Moreover, we need to introduce a restriction operator $R^H$ from $\Sigma^H$ to $\Sigma^0$. As we need to apply this operator to functions of $H^{-1/2}(\Sigma^H)$, $R^H$ has to be defined in a weak sense. We denote by $E^H\in {\cal L}(L^2(\Sigma^0),L^2(\Sigma^H))$  the extension operator by zero from $\Sigma^0$ to $\Sigma^H $:
$$
E^H \phi  = 
\left\{
\begin{array}{ll}
\phi & \quad \mbox{on }\Sigma^0,\\
0 &  \quad\mbox{on } \Sigma^H\setminus\Sigma^0.
\end{array}
\right.
$$
Let $H^{1/2}_{00}(\Sigma^0)$ be the subspace of $H^{1/2}(\Sigma^0)$ defined by:
$$
H^{1/2}_{00}(\Sigma^0) =\{\phi\in H^{1/2}(\Sigma^0) \mid E^H\phi \in H^{1/2}(\Sigma^H)\},
$$ 
and let ${\widetilde H}^{-1/2}(\Sigma^0) = \left(H^{1/2}_{00}(\Sigma^0)\right)'$ be its dual space. The restriction operator $R^H\in {\cal L}(H^{-1/2}(\Sigma^H),{\widetilde H}^{-1/2}(\Sigma^0))$  can then be defined by duality
\begin{multline}\label{eqRestriction}
\qquad
\langle R^H \phi,\psi \rangle_{{\widetilde H}^{-1/2}(\Sigma^0),H^{1/2}_{0,0}(\Sigma^0)} = \langle \phi,E^H\psi \rangle_{H^{-1/2}(\Sigma^H),H^{1/2}(\Sigma^H)},\\
  \forall(\phi,\psi)\in H^{-1/2}(\Sigma^H)\times H^{1/2}_{00}(\Sigma^0).\qquad
\end{multline}
The main result of this section reads as follows.
\begin{theorem}
\label{THsec3_2} 
The operator $R^H \circ \Lambda^H \circ D_{2\pi/3}$ maps $H^{1/2}_{2\pi/3}(\Sigma^i)$ into $H^{-1/2}_{2\pi/3}(\Sigma_0)$ and $\Lambda_{2\pi/3}\in \cL (H^{1/2}_{2\pi/3}(\Sigma^i),H^{-1/2}_{2\pi/3}(\Sigma^i))$ defined by \eqref{Lambda2pi3} admits the factorization
$$\forall \phi\in H^{1/2}_{2\pi/3}(\Sigma^i),\quad\Lambda_{2\pi/3}\,\phi =E_{2\pi/3} \circ R^H \circ \Lambda^H \circ D_{2\pi/3}\phi$$ 
where
\begin{itemize}
	\item $D_{2\pi/3}\in {\cal L}( H^{1/2}(\Sigma^i),H^{1/2}(\Sigma^H))$ is the DtD operator defined by \eqref{eqD2pi3},
	\item $\Lambda^H\in {\cal L}(H^{1/2}(\Sigma^H),H^{-1/2}(\Sigma^H))$ is the half-space DtN operator defined by \eqref{eqLambdaH},
	\item $R^H\in {\cal L}(H^{-1/2}(\Sigma^H),{\widetilde H}^{-1/2}(\Sigma^0))$ is the restriction operator defined by \eqref{eqRestriction},
	\item $E_{2\pi/3} \in {\cal L}(H^{-1/2}_{2\pi/3}(\Sigma_0),H^{-1/2}_{2\pi/3}(\Sigma^i))$ is the extension  operator by symmetry defined in Section \ref{subsect_RestrictionOperators}.
\end{itemize}
\end{theorem}

\begin{proof}
Let $\phi \in H^{1/2}_{2\pi/3}(\Sigma^i).$ From the definition of $D_{2\pi/3}$, the functions $u^e(\phi)\left|_{\Omega^H}\right.$ and $u^H(D_{2\pi/3}\phi)$ satisfy the half-space Helmholtz problem \eqref{PbDemiespsec3} with the same Dirichlet condition on $\Sigma^H$, namely $\psi=D_{2\pi/3} \phi.$ The uniqueness of the solution of this problem implies that $u^e(\phi)\left|_{\Omega^H}\right.=u^H(D_{2\pi/3}\phi)$ and in particular, the traces of their normal derivatives on $\Sigma^0$ coincide, yielding
$$
R_{2\pi/3}\left(-\left.\frac{\partial u^e(\phi)}{\partial \nu^i}\right|_{\Sigma^i}\right) = (R_H \circ \Lambda^H \circ D_{2\pi/3})\phi.
$$
where $R_{2\pi/3}$ is the restriction operator by symmetry defined in Section \ref{subsect_RestrictionOperators}. This relation proves the first part of the theorem. For the second part, we just use that $E_{2\pi/3}$ is the inverse of $R_{2\pi/3}$.
\end{proof}

\noindent In section \ref{SPbdemiesp}, we explain how to compute the half-space DtN operator with the help of an adapted version of the Floquet-Bloch transform defined in \ref{ssub:floquet_bloch_transform} and the resolution of a family of half-space problems (\ref{PbDemiespsec3}) with $k-$quasiperiodic boundary conditions. Section \ref{sec:opextension} deals with the characterization of the DtD operator $D_{2\pi/3}$. The computation of $D_{2\pi/3}$  {\it a priori} requires to compute the solutions $u^e$ of the exterior problem \eqref{Pbext} defined in an unbounded domain. We explain, using the half-space problem and the properties of the problem, how to obtain a characterization of this operator which avoid the solution of the exterior problem. 


\section{Characterization of the half-space DtN operator}
\label{SPbdemiesp}
In this section, we tackle the Dirichlet half-space problem. In other words, for any $\phi\in H^{1/2}(\Sigma^H)$ we want to compute the solution $u^H(\phi)$ in $H^1(\Delta,\Omega^H)$ of
\[
(\mathcal{P}^H) \quad
\left\{
\begin{array}{ll}
\Delta u^H(\phi) + \rho u^H(\phi) =0, &\qquad\hbox{in } \Omega^H, \\[5pt]
u^H(\phi)= \phi, & \qquad\hbox{on }\Sigma^H,
\end{array}
\right.
\]
(see Figure~\ref{notationssec3} for notations). We will deduce a characterization of the half-space DtN operator $\Lambda^H$
\[
\Lambda^H \phi =\left. \frac{\partial u^H(\phi)}{\partial\nu^H}\right|_{\Sigma^H}
\]
where $\nu^H$ is the exterior normal to $\Omega^H$.
\begin{remark}\label{rem:PH_autresinterets}
The half-space problem is not only interesting as a step of our approach to analyze transparent boundary conditions for locally perturbed hexagonal periodic media. Indeed, it also appears naturally in transmission problems between a homogeneous medium and a hexagonal periodic one (see \cite{BonRam02a} where such problems are considered for the case of one dimensional type periodic media and \cite{fliss_cassan} for the case of square lattices.)
\end{remark}

\noindent We develop a method for computing the solution of \eqref{PbDemiespsec3} and the operator $\Lambda^H$, by adapting the method developed in \cite{TheseSonia,FlissAPNUM}. In these works, the half-space solution and then the half-space DtN operator are computed using the Floquet Bloch transform. More precisely, this is done  via the resolution of a family of waveguide problems with quasiperiodic conditions, each waveguide solution being computed thanks to the resolution of elementary cell problems and a stationary Riccati equation. In our case, this approach cannot be directly transposed since 
\begin{itemize}
	\item the waveguide boundary would not correspond to a physical boundary (see the shaded domain of Figure~\ref{fig:omegaw});
	\item but most importantly, computing $D_{2\pi/3}$ would be much more intricate (see Remark \ref{rem:FBT}).
\end{itemize}
Instead, the half-space problem is handled by solving a family of half-space problems with $k-$quasiperiodic boundary conditions. 

\subsection{The half-space problem: from arbitrary data to quasiperiodic data} 
\label{sub:solution_of_the_half-space_problem}
\subsubsection{The Floquet-Bloch (FB) transform and its properties} 
\label{ssub:floquet_bloch_transform}
Following \cite{TheseSonia,Ku01}, we recall below the definition of the FB transform and state without proof its main properties.\\\\
Set ${\mathbb{K}}=\R\times\left(-{\pi}/{L},{\pi}/{L}\right)$. The FB transform of period $L$ (here $L = \sqrt{3}d$) is defined by 
$$\begin{array}{lccl}
\dsp {\cF}:&\cC^\infty_0(\R)&\rightarrow&\dsp L^2_\text{QP}({\mathbb{K}})\\[5pt]
\dsp&\phi(y)&\mapsto&\dsp{\cF}\phi(y;k)=\sqrt{\frac{L}{2\pi}}\sum_{q\in\Z}\phi(y+qL)e^{-\imath q k L}.
\end{array}$$
where $L^2_\text{QP}({\mathbb{K}})$ is the set of functions $\hat{f}\in L^2_{loc}$ such that for any $k\in \left(-{\pi}/{L},{\pi}/{L}\right)$, $\hat{f}(\cdot,k)$ is $k-$ quasi-periodic, that means $\hat{f}(y+nL,k)=\hat{f}(y,k)e^{\imath n k L}$. We equip this space by the norm of $L^2\left({\mathbb{K}_0}\right)$ where $\mathbb{K}_0=\left(-L/2,L/2\right)\times\left(-{\pi}/{L},{\pi}/{L}\right)$.\\\\
The operator ${\cF}$ can be extended as an isometry between $L^2(\R)$ and $L^2_\text{QP}({\mathbb{K}})$:
\[
	\forall \phi,\psi \in L^2(\R),\quad \langle {\cF}\phi,{\cF}\psi \rangle_{L^2_\text{QP}({\mathbb{K}})} =  \langle \phi,\psi \rangle_{L^2(\R)}.
\]
This transform is a privileged tool for the study of PDE with periodic coefficients because it commutes with
\begin{itemize}
	\item any differential operator ;
	\item the multiplication by any periodic function with period $L$.
\end{itemize}
The FB transform is invertible and the inversion formula is given by
\begin{equation}\label{eq:FBext_inv}
\forall y\in\R,\quad \phi(y)=\sqrt{\frac{L}{2\pi}}\int\limits_{-\pi/L}^{\pi/L}{\cF}\phi(y;k)dk.
\end{equation}
\noindent Next we define the partial FB transform in the $y$-direction in the half-space $\Omega^H$:
$$\begin{array}{lccl}
\dsp {\cF}_y:&L^2(\Omega^H)&\rightarrow&\dsp L^2_\text{QP}\left(\Omega^H\times\left(-\frac{\pi}{L},\frac{\pi}{L}\right)\right)\\[7pt]
\dsp&u(x,y)&\mapsto&\dsp
\dsp {\cF}_yu(x,y;k)
\end{array}$$
with 
\[
	\forall x,\quad ({\cF}_yu)(x,\cdot;\cdot) = {\cF}\left[u(x,\cdot)\right],
\]
and $L^2_\text{QP}\left(\Omega^H\times\left(-{\pi}/{L},{\pi}/{L}\right)\right)$ is the set of square integrable functions $\hat{f}$, locally in the $y-$ direction, such that for any $k\in \left(-{\pi}/{L},{\pi}/{L}\right)$, $\hat{f}(\cdot,k)$ is $k-$ quasi-periodic in the $y-$direction, that means $\hat{f}(\cdot,y+nL,k)=\hat{f}(\cdot,y,k)e^{\imath n k L}$. We equip this space by the norm of $L^2\left(\Omega^W\times\left(-{\pi}/{L},{\pi}/{L}\right)\right)$ 
where $\Omega^W = \Omega^H\cap\left\{y\in (-{L}/{2},{L}/{2})\right\}$ (see Figure~\ref{fig:omegaw}).
\begin{figure}[ht]
  \centering
\psfrag{a}[c][][1]{$\Omega^W$}
\psfrag{b}[c][][1]{$\Sigma_{00}^\ell$}
  \includegraphics[width=.6\textwidth]{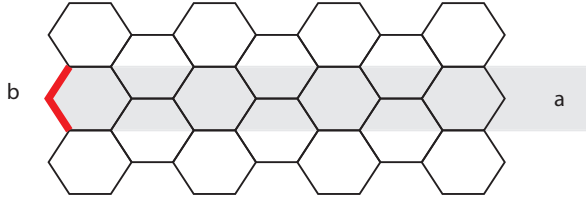}
  \caption{The domain $\Omega^W$ and its left boundary $\Sigma_{00}^\ell$.}
  \label{fig:omegaw}
\end{figure}
~\\\\It is easy to see that the partial FB transform ${\cF}_y$ defines an isomorphism from $L^2(\Omega^H)$ into $L^2_\text{QP}\left(\Omega^H\times(-{\pi}/{L},{\pi}/{L})\right)$.
\\\\
Now we want to know how the Floquet Bloch transform is defined or can be extended to the other functional spaces appearing in our study ($H^1(\Omega^H,\triangle),\,H^{1/2}(\Sigma^H)$ and $H^{-1/2}(\Sigma^H)$). 
To make a rigorous presentation, we need to introduce, $k$ being a parameter between $-\pi/L$ and $\pi/L$,  the so called  ${k}$-quasiperiodic extension operator $E_{k}^{QP}\in \mathcal{L}\left(L^2(\Omega^W),L^2_{unif}(\Omega^H)\right)$ defined by
\[
	\forall u\in L^2(\Omega^W),\;\forall q\in\Z,\;\forall (x,y)\in\Omega^W,\quad E_{k}^{QP}u(x,y+qL)=u(x,y)\,e^{\imath q {k} L},
\]
where $L^2_{unif}(\Omega^H)$ is the normed space defined by
\[
	L^2_{unif}(\Omega^H):=\left\{u\in L^2_{loc}(\Omega^H),\quad\sup_{q\in \Z}\int_{\Omega^W+q{\bf e}_2}\abs{u}^2<+\infty \right\}.
\]
A natural functional space which appears is the set of locally $L^2$, $k-$quasiperiodic functions defined in $\Omega^H$, denoted $L^2_{k}(\Omega^H)$ and characterized by
\[
	L^2_{k}(\Omega^H) = E_{k}^{QP}\left(L^2(\Omega^W)\right).
\]
We can introduce the corresponding $k$-quasiperiodic restriction operator $$R_{k}^{QP}\in \mathcal{L}\left(L^2_k(\Omega^H),L^2(\Omega^W)\right)$$ defined by 
\[
	\forall u_{k}\in 	L^2_{k}(\Omega^H),\quad R_{k}^{QP}u_{k}:=\left.u_{k}\right|_{\Omega^W}.
\]
Noting that
\[
		L^2(\Omega^W) = R_{k}^{QP}\left(L^2_{k}(\Omega^H)\right),
\]
it is easy to see that $L^2_{k}(\Omega^H)$ and $L^2(\Omega^W)$ are isomorphic. We can then consider  $L^2_{k}(\Omega^H)$ as a Hilbert space when it is endowed with the inner product of $L^2(\Omega^W)$. More precisely, we define the scalar product on $L^2_{k}(\Omega^H)$ as follows
\[
	\forall (u_{k},v_{k})\in L^2_{k}(\Omega^H)^2,\quad \left(u_{k},v_{k}\right)_{L^2_{k}(\Omega^H)} = \left(R^{QP}_{k}u_{k},R^{QP}_{k}v_{k}\right)_{L^2(\Omega^W)}.
\]
Next we define smooth quasiperiodic functions in $\Omega^H$ 
$$
\cC^\infty_{k}(\Omega^H) =
\left\{u\in \cC^\infty(\Omega^H),\; \quad u(x,y+{L})=u(x,y)e^{\imath k{L}},\; \forall (x,y) \in \Omega^H\right\}.
$$
and smooth quasiperiodic functions in $\Omega^W$
\begin{equation*}
\cC^\infty_{k}(\Omega^W) = R_{k}^{QP}\left(\cC^\infty_{k}(\Omega^H)\right)
	\end{equation*}
Let $H^1_{k}(\Omega^W)$ (resp. $H^1_{k}(\Delta,\Omega^W)$) be the closure of $\cC^\infty_{k}(\Omega^W)$ in $H^1(\Omega^W)$ (resp. in $H^1(\Delta,\Omega^W)$) equipped with the norm of $H^1(\Omega^W)$ (resp. the norm of $H^1(\Delta,\Omega^W)$) and let $H^1_{k}(\Omega^H)$ (resp. $H^1_{k}(\Delta,\Omega^H)$) be defined by
\begin{equation*}
H^1_{k}(\Omega^H) = E^{QP}_{k}\left(H^1_{k}(\Omega^W)\right)\quad\left(\text{ resp.}\; H^1_{k}(\triangle,\Omega^H) = E^{QP}_{k}\left(H^1_{k}(\Delta,\Omega^W)\right)\right),
	\end{equation*}
that we also equip with the norm of $H^1(\Omega^W)$ (resp. the norm of $H^1(\Delta,\Omega^W)$). 
\begin{remark}\label{rem:H1_ky}
		The functions of $H^1_{k}(\Omega^H)$ are nothing but the $k-$quasiperiodic extensions of functions in $H^1_{k}(\Omega^W)$. Therefore they are $H^1$ in any horizontal strip but not in the vertical ones.
\end{remark}

\noindent Let us denote (see Figure~\ref{fig:omegaw})
\[
	\Sigma_{00}^\ell := \Sigma^H\cap\overline{\Omega^W}\quad\left(\neq \Sigma^0\right).
\]
and let us define $H^{1/2}_{k}(\Sigma_{00}^\ell)$ by
\[
	H^{1/2}_{k}(\Sigma_{00}^\ell) := \left\{\left.u\right|_{\Sigma_{00}^\ell},\;u\in H^1_{k}(\Omega^W)\right\},
\]
equipped with the graph norm. We define the space $H^{1/2}_{k}(\Sigma^H)$ by $k-$quasiperiodic extension of functions of $H^{1/2}_{k}(\Sigma_{00}^\ell)$:
\begin{equation*}
H^{1/2}_{k}(\Sigma^H) = E^{QP}_{k}\left(H^{1/2}_{k}(\Sigma_{00}^\ell)\right),
	\end{equation*}
that we equip with the norm of $H^{1/2}_{k}(\Sigma_{00}^\ell)$.
\\\\
The space $H^{1/2}_{k}(\Sigma_{00}^\ell)$ is a dense subspace of $H^{1/2}(\Sigma_{00}^\ell)$ and the embedding from $H^{1/2}_{k}(\Sigma_{00}^\ell)$ onto $H^{1/2}(\Sigma_{00}^\ell)$ is continuous. We can then define the dual space of $H^{1/2}_{k}(\Sigma_{00}^\ell)$ that we call $H^{-1/2}_{k}(\Sigma_{00}^\ell)$. According to Green's formula,  we can show that
\[
	H^{-1/2}_{k}(\Sigma_{00}^\ell) = \left\{\left.-\frac{\partial u}{\partial\nu}\right|_{\Sigma_{00}^\ell},\;u\in H^{1}_{k}(\Delta,\Omega^W)\right\}.
\]
To define the space of $k$-extension of functions in $H^{-1/2}_{k}(\Sigma_{00}^\ell)$, we need to define the extension operator $E^{QP}_{k}$ on $H^{-1/2}_{k}(\Sigma_{00}^\ell)$ in a weak sense. Actually, this can be done by duality by setting for all $(\psi_0,\phi_{k})\in H^{-1/2}_{k}(\Sigma_{00}^\ell)\times H^{1/2}_{k}(\Sigma^H)$
\[
	\langle E^{QP}_{k}\psi_0,\phi_{k}\rangle_{\left(H^{1/2}_{k}(\Sigma^H)\right)^\prime, H^{1/2}_{k}(\Sigma^H)} = \langle \psi_0,R^{QP}_{k}\phi_{k}\rangle_{H^{-1/2}_{k}(\Sigma_{00}^\ell), H^{1/2}_{k}(\Sigma_{00}^\ell)}.
\]
Finally, we define 
$
	H^{-1/2}_{k}(\Sigma^H)=E^{QP}_{k}\left(H^{-1/2}_{k}(\Sigma_{00}^\ell)\right).
$\\\\
We can now state the following results.
\begin{theorem}\label{th:FB_Hs}
	$\mathcal{F}_y$ is an isomorphism from $X^H$ onto 
	\[
	X_\text{QP}:=\quad \left\{\widehat{u}\in L^2\left(-{\pi}/{L},{\pi}/{L};X^H\right)\,\mid 
		\mbox{ for a. e. } k \in (-{\pi}/{L}, {\pi}/{L} ), \, \widehat{u}(\cdot; k) \in X^H_k \,\right\},
	\]equipped with the norm
	$		\dsp\norm{\widehat{u}}_{X_\text{QP}}^2=\int_{-\pi/L}^{\pi/L}\norm{\widehat{u}(\cdot; k)}_{X^H_k}^2\,dk
	$
	where
\begin{itemize}
	\item $X^H=H^1(\Omega^H)$, $X_\text{QP}=H^1_\text{QP}\left(\Omega^H\times\left(-{\pi}/{L},{\pi}/{L}\right)\right)$ and $X^H_k = H^1_k(\Omega^H)$;
	\item $X^H=H^1(\triangle, \Omega^H)$, $X_\text{QP}=H^1_\text{QP}\left(\triangle;\Omega^H\times\left(-{\pi}/{L},{\pi}/{L}\right)\right)$ and \\$X^H_k = H^1_k(\triangle, \Omega^H)$;
	\item $X^H=H^{1/2}(\Sigma^H)$, $X_\text{QP}=H^{1/2}_\text{QP}\left(\Sigma^H\times\left(-{\pi}/{L},{\pi}/{L}\right)\right)$ and $X^H_k = H^{1/2}_k(\Sigma^H)$.
\end{itemize}
\end{theorem}
\noindent Finally, we can extend by duality the definition of $\mathcal{F}_y$ to the space $H^{-1/2}({\Sigma}^H)$  introducing the dual of $H^{1/2}_\text{QP}\left(\Sigma^H\times(-{\pi}/{L},{\pi}/{L})\right)$
\[ \left|
	\begin{array}{l}
H^{-1/2}_\text{QP}\left(\Sigma^H\times(-{\pi}/{L},{\pi}/{L})\right):= \\[7pt]
\quad \left\{\widehat{\psi}\in L^2\left(-{\pi}/{L},{\pi}/{L};\,X\right)\,\mid
	\forall k \in (-{\pi}/{L},{\pi}/
{L} ), \, \widehat{\psi}(\cdot; k) \in H^{-1/2}_{k}(\Sigma^H) \,\right\},
\end{array} \right.
\]
where $
	X = E^{QP}_0\left(\left(H^{1/2}_{00}(\Sigma^{\ell}_{00})\right)^\prime\right).
$
\\\\
According to Theorem \ref{th:FB_Hs}, the partial FB transform $\mathcal{F}_y$  defines an isomorphism from $H^{1/2}({\Sigma}^H)$ onto $H^{1/2}_\text{QP}\left(\Sigma^H\times\left(-\pi/L,\pi/L\right)\right)$. Using the Riesz representation theorem, the FB transform can then be extended by duality as an isomorphism from $H^{-1/2}({\Sigma}^H)$ onto $H^{-1/2}_\text{QP}\left(\Sigma^H\times\left(-\pi/L,\pi/L\right)\right)$ see \cite{FlissAPNUM} for more details. 
\subsubsection{Application to the half-space problem} 
\label{ssub:application_to_the_half-space_problem}
The above results imply that for almost every $k$ in $\left(-{\pi}/{L},{\pi}/{L}\right)$, we have
\[
\begin{array}{|l}
	\forall \phi\in H^{1/2}(\Sigma^H),\;\widehat{\phi}_{k}:=\cF_y \phi (\cdot;k)\in H^{1/2}_{k}(\Sigma^H)\\[5pt]
	\forall \psi\in H^{-1/2}(\Sigma^H),\; \widehat{\psi}_{k}:=\cF_y \psi (\cdot;k)\in H^{-1/2}_{k}(\Sigma^H)\\[5pt]
	\forall u^H \in H^1(\Delta,\Omega^H),\; \widehat{u}^H_{k} := \cF_y u^H (\cdot;k)\in H^1_{k}(\Delta,\Omega^H)
\end{array}
\]
The following theorem is a direct consequence of the properties of the FB transform given in the previous section.
\begin{theorem}\label{th:FBT_uH}
Let $u^H(\phi)$ be the solution of problem $(\mathcal{P}^H)$ (see equation \eqref{PbDemiespsec3}). For every $k\in (-{\pi}/{L},{\pi}/{L})$, $\widehat{u}^H_{k}\big(\widehat{\phi}_{k}\big):=\mathcal{F}_y\left(u^H(\phi)\right)(\cdot;k)$ is the unique solution $H^1_{k}(\Delta,\Omega^H)$ of the half-space problem with the $k$-quasiperiodic boundary condition
$
	\widehat{\phi}_{k} = \cF_y\phi(\cdot;k)
$.
\end{theorem}
\noindent Using the inversion formula  \eqref{eq:FBext_inv}, we can recover $u^H(\phi)$ in the whole domain $\Omega^H$ for any Dirichlet condition $\phi\in H^{1/2}({\Sigma}^H)$:
\begin{equation}\label{eq:extenduH}
	u^H(\phi) = \sqrt{\frac{L}{2\pi}}\int_{-{\pi}/{L}}^{{\pi}/{L}}\widehat{u}^H_{k}\big(\widehat{\phi}_{k}\big)\;dk.
	\end{equation}
Then we can show the following theorem which expresses that the half-space DtN operator $\Lambda^H$ can be described in terms of a family of ``quasiperiodic'' half-space DtN operators.
\begin{theorem}\label{th:FBT_LambdaH}
The half-space DtN operator $\Lambda^H$ is given by:
\begin{equation}\label{eq:lambdaH}
	\forall\phi\in H^{1/2}(\Sigma^H),\quad
\Lambda^H\phi = \sqrt{\frac{L}{2\pi}}\int_{-{\pi}/{L}}^{{\pi}/{L}}\widehat{\Lambda}_k^H\,\widehat{\phi}_{k}\;dk,
\end{equation}
where $\widehat{\Lambda}_k^H$ is $k-$quasiperiodic half-space DtN operator, defined by
\begin{equation}\label{eq:lambdaH_chapeau}
	\widehat{\Lambda}_k^H\,\widehat{\phi}_{k}=\left.\frac{\partial \widehat{u}^H_{k}\big(\widehat{\phi}_{k}\big)}{\partial \nu^H}\right|_{\Sigma^H}
\end{equation}
\end{theorem}
\noindent According to relations \eqref{eq:extenduH} (resp. \eqref{eq:lambdaH}), the solution of the half-space problem (resp. the half-space DtN operator) for arbitrary boundary data $\phi$ is obtained by superposing the corresponding solutions (resp. DtN operator) for quasiperiodic data. The next subsection is devoted to solving such problems.

\subsection{Solution of the half-space problem for quasiperiodic boundary data}\label{sub:PH_kQPphi}
\noindent Let $k$ be in $(-\pi/L, \pi/L)$, we explain here how to compute the solution of $(\mathcal{P}^H)$ (see \eqref{PbDemiespsec3}) for $k-$quasiperiodic boundary data $\phi: = \widehat{\phi}_k\in H^{1/2}_k(\Sigma^H)$. We have seen in the previous sections (see in particular Theorems \ref{th:FBT_uH} and \ref{th:FBT_LambdaH}) that for any $\widehat{\phi}_k\in H^{1/2}_k(\Sigma^H)$, \eqref{PbDemiespsec3} admits a unique solution $\widehat{u}^H_k(\widehat{\phi}_k)\in H^1_k(\Delta,\Omega^H)$ and $\widehat{\Lambda}^H_k\widehat{\phi}_k\in H^{-1/2}_k(\Sigma^H)$.
\\\\
\noindent This half-space problem is in some sense the counterpart of the waveguide problem with $k-$quasiperiodic conditions used in \cite{TheseSonia,FlissAPNUM} to determine the half-space DtN operator for the case of a square periodicity cell. In \cite{TheseSonia,FlissAPNUM}, the basic tools are the resolution of elementary cell problems and a stationary Riccati equation whose solution is a so-called propagation operator. 
\\\\
Let us begin with some notation. Let $\mathcal{C}_{00}$ be a periodicity cell whose boundary meets the vertical boundary $\Sigma^H$. Given $p\in\mathbb{N},\,q\in\mathbb{Z}$, we introduce the vector ${\mathbf V}_{pq}=p{\bf e}_1+q{\bf e}_2$ (See Assumption (A1) in Section \ref{sec:intro} and Figure~\ref{fig:geometry} for the definition of the directions of periodicity ${\bf e}_1$ and ${\bf e}_2$). The cell $\mathcal{C}_{pq}$ of the periodic half-space can then be defined by translation of the reference cell (see Figure~\ref{phiperiode})
\[
	\forall p\in\mathbb{N},\,q\in\mathbb{Z},\quad \mathcal{C}_{pq} = \mathcal{C}_{00}+{\mathbf V}_{pq}.
\]
We will denote by $\Omega_p$ the vertical ``strip'' containing the cell $\mathcal{C}_{p0}$:
$$\Omega_p=\bigcup_{q\in\mathbb{Z}}\, \mathcal{C}_{pq}.$$
In the following, for a cell of periodicity $\mathcal{C}_{pq}$, we introduce the oriented boundaries described in Figure~\ref{notationscell}. 

\begin{figure}[ht]
  \centering
\psfrag{a}[c][][1.0]{$\mathcal{C}_{pq}$}
\psfrag{b}[c][][1.0]{$\mathcal{C}_{pq}$}
\psfrag{h}[c][][1.0]{$\color{mred}\Sigma_{pq}^{\ell}$}
\psfrag{j}[c][][1.0]{$\color{mred}\Gamma_{pq}^{+}$}
\psfrag{k}[c][][1.0]{$\color{mblue}\Gamma_{pq}^{-}$}
\includegraphics[width=.6\textwidth]{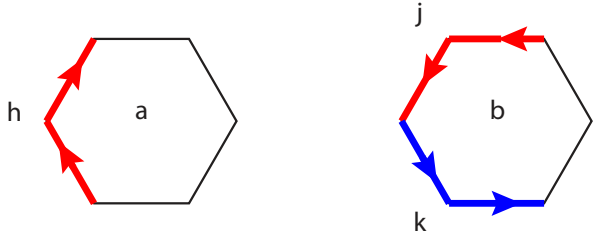}
\caption{Notations used for a periodicity cell}
\label{notationscell}
\end{figure}

\noindent Finally, using the $k$-quasiperiodic restriction operator $R^{QP}_k$ and the $k$-quasiperiodic extension operator $E^{QP}_k$ defined in Section \ref{ssub:floquet_bloch_transform}, we recall that $H^{1/2}_k(\Sigma^H)$ is isomorphic to $H^{1/2}_k(\Sigma_{00}^{\ell})$. 
\subsubsection{The propagation operator} 
\label{ssub:the_propagation_operator}
\noindent We can now introduce the operator $\cP_k$ defined by
\[ 
	\forall \phi^0_k\in H^{1/2}_k(\Sigma_{00}^\ell),\qquad \cP_k\,\phi^0_k =\left. \widehat{u}^H_k(E^{QP}_k\,\phi^0_k)\right|_{\Sigma_{10}^{\ell}}
\] 
where $\widehat{u}^H_k(E^{QP}_k\,\phi^0_k)$ is the unique solution of $(\mathcal{P}^H)$ (see \eqref{PbDemiespsec3}) with boundary condition the $k-$quasiperiodic extension of $\phi^0_k$: $E^{QP}_k\,\phi^0_k$.

\begin{figure}[ht]
\begin{center}
	\psfrag{a}[c][][1.0]{$\mathcal{C}_{00}$}
	\psfrag{c}[c][][1.0]{$\mathcal{C}_{10}$}
	\psfrag{b}[c][][1.0]{$\color{mblue}\phi^0_k$}
	\psfrag{d}[c][][1.0]{$\color{mblue}\cP_k\,\phi^0_k$}
	\includegraphics[width=.3\textwidth]{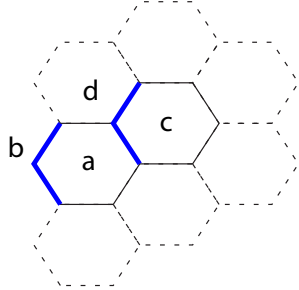}
\end{center}
\caption{Description of the propagation operator}
\label{opP}
\end{figure}

\begin{remark}\label{rem:ident_spaces}
For the sake of simplicity, we will often identify throughout the paper functional spaces of functions acting on $\Sigma_{pq}^{\ell}$, with the same spaces acting on $\Sigma_{00}^{\ell}$. Typically, $H^{1/2}_k(\Sigma_{pq}^{\ell})$ for arbitrary $(p,q)\in\mathbb{N}\times \mathbb{Z}$ will be identified with $H^{1/2}_k(\Sigma_{00}^{\ell})$.
\end{remark}

\noindent The operator $\cP_k$, considered now as a bounded linear operator from $H^{1/2}_k(\Sigma_{00}^{\ell})$ onto itself (see Remark \ref{rem:ident_spaces}), is called a ``propagation operator''. Indeed, the solution can be reconstructed in any cell of the medium from its values in the reference cell $\mathcal{C}_{00}$ and the knowledge of $\cP_k$, as shown in the next result.
\begin{theorem} \label{Thu}
\label{reconsu}
For any $k-$quasiperiodic condition $\widehat{\phi}_k \in H^{1/2}_k(\Sigma^H)$, the solution $\widehat{u}^H_k(\widehat{\phi}_k)$ of \eqref{PbDemiespsec3} is given by
\begin{equation}
\label{uCnm}
\forall p\in\N,~q\in \Z, \quad\left. \widehat{u}^H_k(\widehat{\phi}_k)\right|_{\mathcal{C}_{pq}}=\left.e^{\imath qkL}\widehat{u}^H_k(E^{QP}_k\,\cP_k^p\,R^{QP}_k\widehat{\phi}_k)\right|_{\mathcal{C}_{00}}.
\end{equation}
\end{theorem}
\begin{proof}
First of all, it is clear that the periodicity and the well-posedness of the problem in the vertical direction implies that for any $q \in \Z$, we have
\begin{equation}
\label{uC0q}
\left.\widehat{u}^H_k(\widehat{\phi}_k)\right|_{\mathcal{C}_{0q}}=e^{\imath qkL} \left.\widehat{u}^H_k(\widehat{\phi}_k)\right|_{\mathcal{C}_{00}}.
\end{equation}
In the horizontal direction, due to the periodicity and the well-posedness of problem \eqref{PbDemiespsec3}, we note that the value of $u$ on the cell $\mathcal{C}_{10}$ is related to the one on the cell $\mathcal{C}_{00}$ as follows
\begin{equation}
\label{uC10}
\left.\widehat{u}^H_k(\widehat{\phi}_k)\right|_{\mathcal{C}_{10}}=\left.\widehat{u}^H_k(E^{QP}_k\,\cP_k\,R^{QP}_k\,\widehat{\phi}_k)\right|_{ \mathcal{C}_{00}}.
\end{equation}
By induction, one easily gets that for all $p \in \N$ there holds
\begin{equation}
\label{uCn0}
\left.\widehat{u}^H_k(\widehat{\phi}_k)\right|_{\mathcal{C}_{p0}}=\left.\widehat{u}^H_k(E^{QP}_k\,\cP_k^p\,R^{QP}_k\,\widehat{\phi}_k)\right|_{\mathcal{C}_{00}}.
\end{equation}
Combining relations \eqref{uC0q} and \eqref{uCn0} yields the claimed result \eqref{uCnm}.
\end{proof}

\noindent The next result collects some useful properties of the propagation operator $\cP_k$.
\begin{corollary}\label{cor:prop_Pk}
	The operator $\cP_k\in\mathcal{L}(H^{1/2}_k(\Sigma_{00}^{\ell}) )$ is a compact operator with spectral radius strictly less than one.
\end{corollary}
\begin{proof}
	The compactness of $\cP_k\in\mathcal{L}(H^{1/2}_k(\Sigma_{00}^{\ell}) )$ follows easily from  interior regularity and Sobolev compactness embedding arguments.
	\\\\ By definition of the space $H^1_k(\triangle,\Omega^H)$ (see Remark~\ref{rem:H1_ky}), $\widehat{u}^H_k(\widehat{\phi}_k)$ satisfies in particular for any $\widehat{\phi}_k\in H^{1/2}_k(\Sigma^H)$ 
	\[
		\dsp\int_{\Omega^W}\abs{\widehat{u}^H_k(\widehat{\phi}_k)}^2 \, <+\infty.
	\]
	Due to the $k$-quasiperiodicity of $\widehat{u}^H_k(\widehat{\phi}_k)$, the above relation is equivalent to
	\[
		\dsp\int_{\cup_{p\in\N}\,\cC_{p0}}\abs{\widehat{u}^H_k(\widehat{\phi}_k)}^2 \, <+\infty
	\]
	Moreover, using Theorem \ref{Thu}, we have
	\[\begin{array}{rcl}
		\dsp\int_{\cup_{p\in\N}\,\cC_{p0}}\abs{\widehat{u}^H_k(\widehat{\phi}_k)}^2 \, &=&\dsp \sum_{p\in\N}\int_{\cC_{p0}}\abs{\widehat{u}^H_k(\widehat{\phi}_k)}^2 \, \\[5pt]
		&=&\dsp \sum_{p\in\N}\int_{\cC_{00}}\abs{\widehat{u}^H_k(E^{QP}_k\,\cP_k^p\,R^{QP}_k\widehat{\phi}_k)}^2\,.
	\end{array}
	\]
	Consequently, if $\lambda$ is eigenvalue of $\cP_k$ and $\varphi$ an associated eigenvector, combining the last two relations for $ \widehat{\phi}_k=\varphi$ shows that
	\[
		\left(\sum_{p\in\N}\abs{\lambda}^p\right)\int_{\cC_{00}}\abs{\widehat{u}^H_k(\varphi)}^2\,<+\infty
	\]
	from which we get that
	\[
		\abs{\lambda}<1.
	\]
\end{proof}

\noindent According to Theorem \ref{Thu}, the solution of the half-space problem \eqref{PbDemiespsec3} for a $k-$quasi\-periodic Dirichlet condition is completely determined on the whole domain $\Omega^H$ as soon as it is known on the reference periodicity cell $\mathcal{C}_{00}$, provided the propagation operator $\mathcal{P}_k$ is also known.

\subsubsection{Elementary problems} 
\label{ssub:elementary_cell_problems}

Like in \cite{TheseSonia,FlissAPNUM}, introducing elementary problems allows us to restrict the half-space problem for quasiperiodic boundary data to the determination of the propagation operator $\cP_k$. 
\\\\
More precisely, given $\phi^0_k\in H^{1/2}_k(\Sigma^{\ell}_{00})$, let us introduce the solutions of the following problems set in the vertical strip $\Omega_0=\bigcup_{q\in\mathbb{Z}}\, \mathcal{C}_{0q}$

\begin{itemize}
	\item $E^\ell_k(\phi^0_k)\in H^1_k(\Delta,\Omega_0)$ is the unique solution of (see Figure~\ref{Figpbband})
		\begin{equation}
		\label{CellPb1}
		\left\{
		\begin{array}{ll}
		\Delta E^\ell_k(\phi^0_k) + \rho E^\ell_k(\phi^0_k)  =0,&\quad\hbox{in } \Omega_0, \\[5pt]
		E^\ell_k(\phi^0_k)=\phi^0_k,&\quad\hbox{on }\Sigma_{00}^{\ell} ,\\[5pt]
		E^\ell_k(\phi^0_k)=0,&\quad\hbox{on }\Sigma_{10}^{\ell},
		\end{array}
		\right.
		\end{equation}
	\item 	$E^r_k(\phi^0_k)\in H^1_k(\Delta,\Omega_0)$ is the unique solution of (see Figure~\ref{Figpbband})
			\begin{equation}
			\label{CellPb2}
			\left\{
			\begin{array}{ll}
			\Delta E^r_k(\phi^0_k) + \rho E^r_k(\phi^0_k)  =0,&\quad\hbox{in } \Omega_0, \\[5pt]
			E^r_k(\phi^0_k)=0,&\quad\hbox{on }\Sigma_{00}^{\ell} ,\\[5pt]
			E^r_k(\phi^0_k)=\phi^0_k,&\quad\hbox{on }\Sigma_{10}^{\ell}.
			\end{array}
			\right.
			\end{equation}
\end{itemize}
Let us emphasize that the boundary conditions are described only for the two left lateral sides of  ${\mathcal C}_{00}$ and ${\mathcal C}_{10}$, as the boundary condition on $\Sigma_{0q}^{\ell}$ (resp. on $\Sigma_{1q}^{\ell}$) for $q\in\Z^*$ follows directly from the one on $\Sigma_{00}^{\ell}$ (resp. on $\Sigma_{10}^{\ell}$)  by $k-$quasiperiodicity. 
\begin{figure}[ht]
\begin{center}
\begin{tabular}{lll}
\psfrag{a}[c][][1.0]{$\mathcal{C}_{00}$}
\psfrag{b}[r][][1.0]{\color{mblue}$\left.E_k^\ell\right|_{\Sigma_{00}^\ell}=\phi_k^0$}
\psfrag{c}[l][][1.0]{\color{mblue}$\left.E_k^\ell\right|_{\Sigma_{10}^\ell}=0$}
  \includegraphics[width=.15\textwidth]{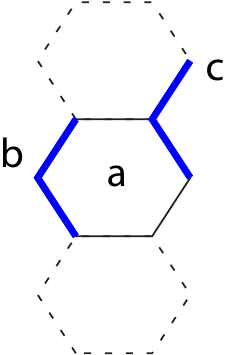}
&\qquad \qquad \qquad\qquad\qquad\qquad&
\psfrag{a}[c][][1.0]{$\mathcal{C}_{00}$}
\psfrag{b}[r][][1.0]{\color{mblue}$\left.E_k^r\right|_{\Sigma_{00}^\ell}=0$}
\psfrag{c}[l][][1.0]{\color{mblue}$\left.E_k^r\right|_{\Sigma_{10}^\ell}=\phi_k^0$}
  \includegraphics[width=.15\textwidth]{Figs/fig11.pdf}
\end{tabular}
\end{center}
\caption{The two strip problems for $E_k^\ell$ (on the left) and $E_k^r$ (on the right)}
\label{Figpbband}
\end{figure}
~\\\\ We also introduce the elementary cell solutions
\[
	e^\ell_k(\phi^0_k) = \left.E^\ell_k(\phi^0_k)\right|_{\mathcal{C}_{00}}\qquad\text{and}\qquad e^r_k(\phi^0_k) = \left.E^r_k(\phi^0_k)\right|_{\mathcal{C}_{00}}.
\]
Conversely, it is clear that $E^{\ell}_k(\phi^0_k)$ and $E^r_k(\phi^0_k)$ are uniquely determined by the above elementary cell solutions due to their quasiperiodicity. Therefore, in practice, one simply needs to solve the corresponding cell problems set in the reference periodicity cell $\cC_{00}$. One might thus wonder why we have introduced the strip problems and not directly the cell problems. In fact, it turns out that viewing the elementary cell solutions as restrictions of the strip problems leads to more compact and simpler expressions in the following.
\\\\
The main advantage of these elementary problems lies in the fact that, by linearity of \eqref{PbDemiespsec3}, one has 
\[
\left.\widehat{u}^H_k(\widehat{\phi}_k)\right|_{\Omega_0}=E^\ell_k(R^{QP}_k\,\widehat{\phi}_k)+E^r_k(\cP_kR^{QP}_k\widehat{\phi}_k),
\]
and then in the reference cell
\begin{equation}
\label{uC00}
\left. \widehat{u}^H_k(\widehat{\phi}_k)\right|_{\mathcal{C}_{00}}=e^\ell_k(\phi_k^0)+e^r_k(\cP_k\phi_k^0).
\end{equation}

\subsubsection{The Ricatti equation for the determination of the propagation operator} 
\label{ssub:the_ricatti_equation_for_the_determination_of_the_propagator}

Assuming the elementary cell solutions are known, it remains to determine the propagation operator $\cP_k.$ To this end, we use the same strategy as in \cite{TheseSonia}. In short, the equation characterizing the propagation operator $\cP_k$ is obtained by writing the continuity of the normal derivative of $\widehat{u}^H_k$ across each boundary $\Sigma_{pq}^{\ell}$, which is ensured by Theorem \ref{Thu}.
\\\\
To this end, we introduce four local DtN operators associated to the elementary problems \eqref{CellPb1} and \eqref{CellPb2}. We refer the reader to Section \ref{ssub:floquet_bloch_transform} for the definition of the spaces $H^{1/2}_k(\Sigma^{\ell}_{00})$ and $H^{-1/2}_k(\Sigma^{\ell}_{00})$.
\begin{definition}
\label{defDtN}
We introduce the following local DtN operators
\[
 \mathcal{T}^{ij}_k\in \mathcal{L}\left(H^{1/2}_k(\Sigma^{\ell}_{00}), H^{-1/2}_k(\Sigma^{\ell}_{00})\right),\qquad  i,j\in\{\ell,r\},
\]
where for all $\phi_k^0 \in H^{1/2}_k(\Sigma^{\ell}_{00})$:
\[
	\begin{array}{ccc}
		 \mathcal{T}^{\ell\ell}_k\phi_k^0 = \left.\nabla E^\ell_k(\phi_k^0)\cdot \nu \,\right|_{\Sigma_{00}^{\ell}}&\qquad&\mathcal{T}^{\ell r}_k\phi_k^0 = \nabla E^\ell_k(\phi_k^0)\cdot \nu \,\left|_{\Sigma_{10}^{\ell}}\right.\\[5pt]
		\mathcal{T}^{r\ell}_k\phi_k^0 = \nabla E^r_k(\phi_k^0)\cdot \nu \,\left|_{\Sigma_{00}^{\ell}}\right.&\qquad&\mathcal{T}^{r r}_k\phi_k^0 = \nabla E^r_k(\phi_k^0)\cdot \nu \,\left|_{\Sigma_{10}^{\ell}}\right.
	\end{array}
\]
where $\nu$ is the outgoing unit normal to $\mathcal{C}_{00}$.
\end{definition}

\begin{figure}[ht]
\begin{center}
\begin{tabular}{lll}
\psfrag{a}[c][][0.8]{$E_k^\ell(\phi_k^0)$}
\psfrag{b}[r][][1.0]{\color{mblue}$\phi_k^0$}
\psfrag{c}[l][][1.0]{\color{mblue}$0$}
\psfrag{d}[c][][0.8]{$E_k^\ell(\phi_k^0)$}
\psfrag{e}[r][][1.0]{\color{mred}$\mathcal{T}^{\ell\ell}_k\phi_k^0$}
\psfrag{f}[l][][1.0]{\color{mred}$\mathcal{T}^{\ell r}_k\phi_k^0$}
  \includegraphics[width=.4\textwidth]{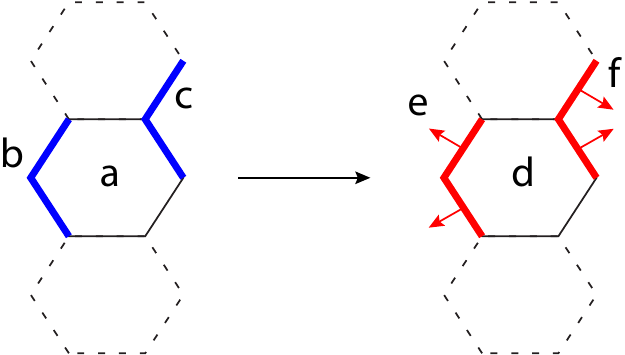}
&\quad \quad &
\psfrag{a}[c][][0.8]{$E_k^r(\phi_k^0)$}
\psfrag{c}[l][][1.0]{\color{mblue}$\phi_k^0$}
\psfrag{b}[r][][1.0]{\color{mblue}$0$}
\psfrag{d}[c][][0.8]{$E_k^r(\phi_k^0)$}
\psfrag{e}[r][][1.0]{\color{mred}$\mathcal{T}^{r\ell}_k\phi_k^0$}
\psfrag{f}[l][][1.0]{\color{mred}$\mathcal{T}^{r r}_k\phi_k^0$}
  \includegraphics[width=.4\textwidth]{Figs/fig20.pdf}
\end{tabular}
\end{center}
\caption{The four local DtN operators}
\label{fig:DtN_locaux}
\end{figure}

\noindent The characterization of the operator $\cP_k$ is then given by the following result.
\begin{theorem}
\label{ThP}
The propagation operator $\cP_k$ is the unique compact operator of $\cL(H^{1/2}_k(\Sigma^{\ell}_{00}))$ with spectral radius strictly less than 1 solution of the stationary Riccati equation
\begin{equation}
\label{Riccati}
\mathcal{T}^{r\ell}_k\cP_k^2+(\mathcal{T}^{\ell \ell}_k+\mathcal{T}^{rr}_k)\cP_k + \mathcal{T}^{\ell r}_k=0.
\end{equation}
\end{theorem}
\begin{proof}
The proof involves two steps: we first show that $\cP_k$ satisfies \eqref{Riccati} and then we prove that \eqref{Riccati} has a unique solution with spectral radius strictly less than 1.\\
\\Step 1: $\cP_k$ satisfies \eqref{Riccati}. 
First of all, according to Corollary \ref{cor:prop_Pk}, $\cP_k$ is an operator in $\cL(H^{1/2}_k(\Sigma^{\ell}_{00}))$ with spectral radius strictly less than 1. Moreover, using the continuity of the normal derivative of $\widehat{u}^H_k$, we obtain using equations \eqref{uC10} and \eqref{uC00}, that for any $\phi_k^0\in H^{1/2}_k(\Sigma^{\ell}_{00})$
$$\left.\left(\nabla E^\ell_k(\phi_k^0)+ \nabla E^r_k(\cP_k\phi_k^0)\right)\cdot \nu\right|_{ \Sigma_{10}^{\ell} }= -
   \left.\left(\nabla E^\ell_k(\cP_k\phi_k^0)+ \nabla E^r_k(\cP_k^2\phi_k^0)\right)\cdot \nu\right|_{ \Sigma_{00}^{\ell} },
$$
which gives
$$\mathcal{T}^{\ell r}_k(\phi_k^0)+\mathcal{T}^{rr}_k(\cP_k\phi_k^0)=-\mathcal{T}^{\ell \ell}_k(\cP_k\phi_k^0)-\mathcal{T}^{r\ell}_k(\cP_k^2\phi_k^0).
$$
Since $\phi_k^0\in H^{1/2}_k(\Sigma^{\ell}_{00})$ is arbitrary, we get
$$\mathcal{T}_k^{\ell r}+(\mathcal{T}_k^{rr}+\mathcal{T}_k^{\ell \ell})\cP_k+\mathcal{T}_k^{r\ell}\cP_k^2=0,$$
which is exactly \eqref{Riccati}.\\\\
Step 2: Uniqueness for \eqref{Riccati}. This result follows immediately from the uniqueness of the solution of \eqref{PbDemiespsec3}. Indeed, assume that $\cP\in\cL(H^{1/2}_k(\Sigma^{\ell}_{00}))$  has a spectral radius strictly less than 1 and satisfies
\[
\mathcal{T}^{r\ell}_k\cP^2+(\mathcal{T}^{\ell \ell}_k+\mathcal{T}^{rr}_k)\cP + \mathcal{T}^{\ell r}_k=0.
\]
Then, one can easily show that for any $\widehat{\phi}_k\in H^{1/2}_k(\Sigma^H)$ , the function $v$ defined in each cell ${\cC_{pq}}$, $p\in\N,~q\in \Z$, by 
\[
	v\left|_{\cC_{pq}}\right.=e^{\imath qkL}e^\ell_k(\cP^p\,R^{QP}_k\,\widehat{\phi}_k)+e^r_k(\cP^{p+1}\,R^{QP}_k\,\widehat{\phi}_k),
\]
belongs to $H^1_k(\Delta,\Omega^H)$ (thanks to the definition of $e^{\ell}_k$ and $e^r_k$ and using the fact that $\cP$ is a solution of the Riccati equation with spectral radius strictly less than 1) and solves \eqref{PbDemiespsec3}. Then, the uniqueness of the solution of \eqref{PbDemiespsec3} in $H^1_k(\Delta,\Omega^H)$ implies that
\[
	v=\widehat{u}^H_k(\widehat{\phi}_k),
\]
and by definition of $v$ and $\cP_k$, we deduce that $\cP=\cP_k$.
\end{proof}
~\\
By Theorem \ref{Thu} and expression \eqref{uC00}, solving the elementary problems \eqref{CellPb1}-\eqref{CellPb2} and the Riccati equation \eqref{Riccati} allows us to reconstruct cell by cell the unique solution of $(\mathcal{P}^H)$ (see \eqref{PbDemiespsec3}) in the case of quasiperiodic boundary data.
\\\\
Finally, we deduce from the above analysis the expression of the DtN operator for quasiperiodic boundary condition: 
\begin{proposition}[DtN operator for quasiperiodic boundary data]\label{prop:DtN_k}
	Suppose that the data $\widehat{\phi}_k\in H^{1/2}_k(\Sigma^H)$, the DtN operator is given by
	\begin{equation}\label{eq:Lambda_H_chapeau}
	\widehat{\Lambda}^H_k\widehat{\phi}_k = E^{QP}_k\,\mathcal{T}_k^{\ell \ell}\,R^{QP}_k\,\widehat{\phi}_k + E^{QP}_k\,\mathcal{T}_k^{r \ell}\cP_k\,R^{QP}_k\,\widehat{\phi}_k.
\end{equation}
\end{proposition}

\subsubsection{Additional tools} 
\label{ssub:additional_tools}
\noindent We conclude this subsection by introducing four local Dirichlet-to-Dirichlet (DtD) operators associated to the cell $\mathcal{C}_{00}$ that will be needed in the sequel.
\begin{definition}
\label{defDtD}
We define the local DtD operators by setting for all $ \phi_k^0 \in H^{1/2}_k(\Sigma^{\ell}_{00})$:
\[
	\begin{array}{ccc}
		 \mathcal{D}^{\ell+}_k\phi_k^0 =  E^\ell_k(\phi_k^0)\left|_{\Gamma_{00}^{ +}}\right.&\qquad\text{and}\qquad& \mathcal{D}^{\ell-}_k\phi_k^0 =  E^\ell_k(\phi_k^0)\left|_{\Gamma_{00}^{ -}},\right.\\[5pt]
		\mathcal{D}^{r+}_k\phi_k^0 =  E^r_k(\phi_k^0)\left|_{\Gamma_{00}^{ +}}\right.&\qquad\text{and}\qquad&\mathcal{D}^{r-}_k\phi_k^0 =  E^r_k(\phi_k^0)\left|_{\Gamma_{00}^{-}}.\right.
	\end{array}
\]
Identifying $\Sigma^{\ell}_{00}$, $\Gamma_{00}^{-}$ and $\Gamma_{00}^{+}$, the above operators will be considered as bounded linear operators from $H^{1/2}_k(\Sigma^{\ell}_{00})$ onto $H^{1/2}(\Sigma^{\ell}_{00})$.
\end{definition}

\begin{figure}[ht]
\begin{center}
\begin{tabular}{lll}
\psfrag{a}[c][][0.8]{$E_k^\ell(\phi_k^0)$}
\psfrag{b}[r][][1.0]{\color{mblue}$\phi_k^0$}
\psfrag{c}[l][][1.0]{\color{mblue}$0$}
\psfrag{d}[c][][0.8]{$E_k^\ell(\phi_k^0)$}
\psfrag{e}[c][][1.0]{\color{mblue}$\mathcal{D}^{\ell +}_k\phi_k^0$}
\psfrag{f}[c][][1.0]{\color{mblue}$\mathcal{D}^{\ell -}_k\phi_k^0$}
  \includegraphics[width=.4\textwidth]{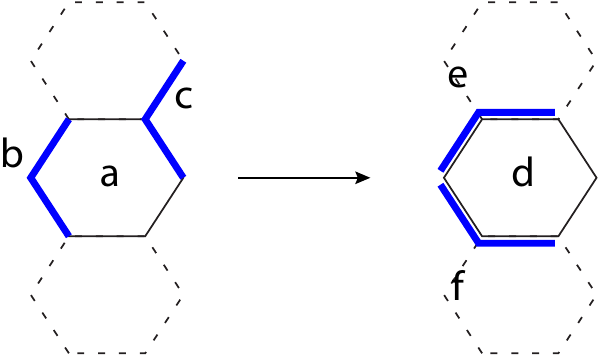}
&\quad \quad &
\psfrag{a}[c][][0.8]{$E_k^r(\phi_k^0)$}
\psfrag{c}[l][][1.0]{\color{mblue}$\phi_k^0$}
\psfrag{b}[r][][1.0]{\color{mblue}$0$}
\psfrag{d}[c][][0.8]{$E_k^r(\phi_k^0)$}
\psfrag{e}[c][][1.0]{\color{mblue}$\mathcal{D}^{r+}_k\phi_k^0$}
\psfrag{f}[c][][1.0]{\color{mblue}$\mathcal{D}^{r -}_k\phi_k^0$}
  \includegraphics[width=.4\textwidth]{Figs/fig22.pdf}
\end{tabular}
\end{center}
\caption{The four local DtD operators}
\label{fig:DtD_locaux}
\end{figure}

\noindent Using Theorem \ref{Thu} and expression \eqref{uC00} we obtain the following result.
\begin{corollary}\label{cor:trace_uH}
For any $\widehat{\phi}_k\in H^{1/2}_k(\Sigma^H)$, for any $p\in\N$ and $q\in\Z$, we have
\[
\widehat{u}^H_k(\widehat{\phi}_k)\left|_{\Gamma_{pq}^{+} }\right.=e^{\imath q kL}\left[\mathcal{D}_k^{\ell +}  \, \cP^p_k\,R^{QP}_k\,\widehat{\phi}_k + \mathcal{D}_k^{r +}\,\cP_k^{p+1}\,R^{QP}_k\,\widehat{\phi}_k\right],
\]
\[
\widehat{u}^H_k(\widehat{\phi}_k)\left|_{\Gamma_{pq}^{-} }\right.=e^{\imath q kL}\left[\mathcal{D}_k^{\ell -} \, \cP^p_k\,R^{QP}_k\,\widehat{\phi}_k + \mathcal{D}_k^{r -} \, \cP_k^{p+1}\,R^{QP}_k\,\widehat{\phi}_k\right].
\]
\end{corollary}

\section{Characterization of the DtD operator}\label{sec:opextension}
This section is devoted to the determination of the DtD operator $D_{2\pi/3}$. We first show in Section \ref{sub:the_affine_equation} that it solves an affine equation involving an operator $D^H$ associated with the half-space problem. Using the FB transform, we derive a semi-analytic expression for $D^H$ in Section \ref{sub:DH_def}. Finally, we deduce in Section \ref{sub:towards_the_integral_equation} an equivalent integral formulation of the affine equation 
which is more suitable for future numerical approximation.
\subsection{The affine equation} 
\label{sub:the_affine_equation}

\begin{figure}[ht]
  \centering

\psfrag{a}[l][][1.0]{$\Sigma^{0}$}
\psfrag{b}[l][][1.0]{$\Sigma^+$}
\psfrag{c}[c][][1.0]{$\Sigma^-_{2\pi/3}$}
\psfrag{d}[c][][1.0]{$\Sigma^+_{2\pi/3}$}
\psfrag{e}[l][][1.0]{$\Sigma^-$}
\psfrag{f}[c][][1.0]{$\mathcal{C}_{00}$}
  \includegraphics[width=.4\textwidth]{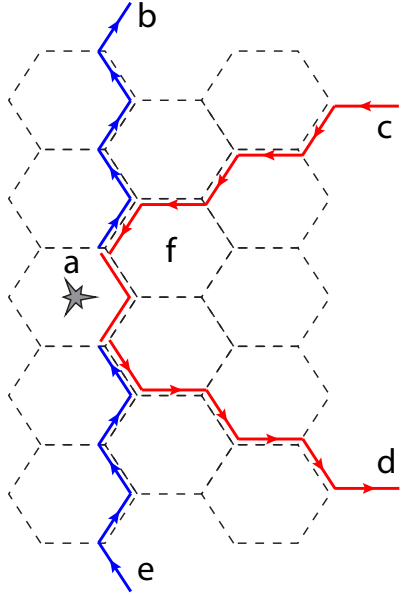}
  \caption{Symmetry axis}
  \label{notations_SigmaH}
\end{figure}

\noindent First of all let us introduce some useful notation (see Figure~\ref{notations_SigmaH}): $\Sigma^H=\Sigma^- \cup \Sigma^0 \cup \Sigma^+$, $\Sigma^-_{2\pi/3}=\Theta_{2\pi/3} \Sigma^-$, $\Sigma^+_{2\pi/3}=\Theta_{-2\pi/3} \Sigma^+$, where the center of the rotations $\Theta_{\pm2\pi/3}$ is the center of $\Omega^i$. Let us remind that $\Omega^H$ is the half-space at the right of $\Sigma^H$. \\\\
Let us recall that $D_{2\pi/3}$ is defined by (see Figure~\ref{fig:d2pis3})
\begin{center}
\begin{tabular}{rcc}
$D_{2\pi/3}: H^{1/2}(\Sigma^i)$ & $\longrightarrow$ & $H^{1/2}(\Sigma^H)$ \\
$\phi$ & $\longmapsto$ & $u^e(\phi)\left|_{\Sigma^H}\right. .$
\end{tabular}
\end{center}
We first remark that $D_{2\pi/3}$ belongs to the affine space:
$$\cL_{\Sigma^0}=\left\{L \in \cL(H^{1/2}_{2\pi/3}(\Sigma^i),H^{1/2}(\Sigma^H)),\;~\forall \phi\in H^{1/2}_{2\pi/3}(\Sigma^i) \quad~\left.L\phi\right|_{\Sigma^0}=\left.\phi\right|_{\Sigma^0} \right\}.$$
Let us introduce, $D^H$, called the half-space DtD operator, defined by
\begin{center}
\begin{tabular}{rcl}
$D^H: H^{1/2}(\Sigma^H)$ & $\longrightarrow$ & $H^{1/2}(\Sigma^H)$ \\[5pt]
$\psi$ & $\longmapsto$ & $ \left\{ \begin{array}{ll}\left. D^H \psi\right|_{\Sigma^-}\equiv \left.u^H(\psi)\right|_{\Sigma^-_{2\pi/3}} \\[6pt]  \,\left.D^H \psi\right|_{\Sigma^0}\equiv \left.u^H(\psi)\right|_{\Sigma_{0}} \\[5pt] 
\left.D^H \psi\right|_{\Sigma^+}\equiv \left.u^H(\psi)\right|_{\Sigma^+_{2\pi/3}}
\end{array}\right. $
\end{tabular}
\end{center}
where we have identified $\Sigma^-$ and $\Sigma^-_{2\pi/3},~\Sigma^+$ and $\Sigma^+_{2\pi/3}$ taking into account the directions shown in Figure~\ref{notations_SigmaH}.
\begin{remark}
The range of $D^H$ is included in $H^{1/2}(\Sigma^H)$ because for any $\psi\in H^{1/2}(\Sigma^H)$, $D^H\psi$ is nothing but the trace of the $H^1$ function $u^H(\psi)$ on the broken line $\Sigma^{-}_{2\pi/3}\cup\Sigma^0\cup\Sigma^+_{2\pi/3}$ (identifying $\Sigma^H$ with $\Sigma^-_{2\pi/3}\cup\Sigma^0\cup\Sigma^+_{2\pi/3}$).
\end{remark}

\noindent We have then the following fundamental theorem
\begin{theorem}
\label{THsec3_3}
The operator $D_{2\pi/3}$ is the unique solution of the problem: 
\begin{equation}
\label{PbD}
\mbox{Find } D \in \cL_{\Sigma^0} \mbox{ such that } D=D^H\circ D.
\end{equation} 
\end{theorem}
\begin{remark}
Note that since $\cL_{\Sigma^0}$ is an affine space, the problem \eqref{PbD} is an affine problem, even though the equation is linear.
\end{remark}

\begin{proof}
\underline{Existence:} We prove that the operator $D_{2\pi/3}$ is a solution of \eqref{PbD}. We have already seen that
\begin{equation}
\label{PbD2}
u^e(\phi)\left|_{\Omega^H}\right.=u^H(D_{2\pi/3} \phi).
\end{equation}
Moreover, since $\phi$ is in $H^{1/2}_{2\pi/3} (\Sigma^i),~u^e(\phi)$ is in $H^1_{2\pi/3}(\Omega^e).$ In particular,
\begin{itemize}
	\item $\Sigma^-_{2\pi/3}=\Theta_{2\pi/3} \Sigma^-$ which implies $u^e(\phi)|_{\Sigma^-_{2\pi/3}} \equiv u^e(\phi)|_{\Sigma^-}$ and then \\ $u^H(D_{2\pi/3}\phi)|_{\Sigma^-_{2\pi/3}}=u^e(\phi)|_{\Sigma^-}$ using \eqref{PbD2},
	\item $u^e(\phi)|_{\Sigma^0}=\phi|_{\Sigma^0},$
	\item $\Sigma^+=\Theta_{2\pi/3} \Sigma^+_{2\pi/3}$ which implies $u^e(\phi)|_{\Sigma^+}\equiv u^e(\phi)|_{\Sigma^+_{2\pi/3}} $ and then $u^e(\phi)|_{\Sigma^+}=u^H(D_{2\pi/3}\phi)|_{\Sigma^+_{2\pi/3}}$ using \eqref{PbD2}.
\end{itemize}
Using the definition of $D^H,$ we obtain that $D_{2\pi/3}$ solves \eqref{PbD}.
\\\\
\underline{Uniqueness:} Let $D$ be an operator from $H^{1/2}_{2\pi/3}(\Sigma^i)$ into $H^{1/2}(\Sigma^H)$ such that for all $\phi \in H^{1/2}_{2\pi/3}(\Sigma^i),~D\phi|_{\Sigma^0}=0$ and which satisfies 
\begin{equation}
\label{PbD3}
D^H \circ D - D =0.
\end{equation}
We prove that $D=0.$ Let $\phi \in H^{1/2}_{2\pi/3}(\Sigma^i)$ and $v_0=u^H(D\phi)$ defined in $\Omega^H.$ We have in particular
$$v_{0}\left|_{\Sigma^0}\right.=u^H(D\phi)\left|_{\Sigma^0}\right.=D\phi\left|_{\Sigma^0}\right.=0.$$
Now we build a function in the half-space $\Omega_{2\pi/3}=\Theta_{2\pi/3} \Omega^H$ by: $v_{2\pi/3}=v_0(\Theta_{2\pi/3}x).$ By a classical argument, since $v_0$ is solution of \eqref{PbDemiespsec3} in $\Omega^H,$ it is clear that $v_{2\pi/3}$ is solution of 
$$\Delta v_{2\pi/3} + \rho v_{2\pi/3} = 0,~~~\hbox{in}~~\Omega_{2\pi/3},$$
while $v_0|{\Sigma^0}=0$ implies $v_{{2\pi/3}}\left|_{\Theta_{2\pi/3}\Sigma^0}\right.=0.$ We are going to show that $v_0$ and $v_{2\pi/3}$ coincide in $\Omega^H \cap (\Theta_{2\pi/3}\Omega^H).$ The difference $d_{2\pi/3}=v_0-v_{2\pi/3}$ satisfies
\begin{equation}
\label{PbD4}
\Delta d_{2\pi/3} + \rho d_{2\pi/3} = 0,~~~\hbox{in}~~\Omega^H \cap (\Theta_{2\pi/3}\Omega^H),
\end{equation}
with boundary condition: 
\[
d_{2\pi/3}\left|_{\Sigma^+}\right.=v_0\left|_{\Sigma^+}\right.-v_{2\pi/3}\left|_{\Sigma^+}\right.=v_0\left|_{\Sigma^+}\right.-v_0\left|_{\Sigma^+_{2\pi/3}}\right.,
\]
\[
d_{2\pi/3}\left|_{\Sigma^-_{2\pi/3}}\right.=v_0\left|_{\Sigma^-_{2\pi/3}}\right.-v_{2\pi/3}\left|_{\Sigma^-_{2\pi/3}}\right.=v_0\left|_{\Sigma^-_{2\pi/3}}\right.-v_0\left|_{\Sigma^-}\right.,
\]
since $\Sigma^+=\Theta_{2\pi/3} \Sigma^+_{2\pi/3}$, $\Sigma^-_{2\pi/3}=\Theta_{2\pi/3} \Sigma^-$ and $v_{2\pi/3}=v_0(\Theta_{2\pi/3}).$
Using the definition of $v_0,$ we have
$$d_{2\pi/3}\left|_{\Sigma^+}\right. = u^H(D\phi)\left|_{\Sigma^+}\right.-u^H(D\phi)\left|_{\Sigma^+_{2\pi/3}}\right.,$$
which gives using the definitions of $u^H$ and $D^H$
$$d_{2\pi/3}\left|_{\Sigma^+}\right. = D\phi\left|_{\Sigma^+}\right.-D^H \circ D\phi\left|_{\Sigma^+}\right.,$$
and then by \eqref{PbD3}
\begin{equation}
\label{PbD6}
d_{2\pi/3}\left|_{\Sigma^+}\right.=0.
\end{equation}
In the same way, we have :
$$d_{2\pi/3}\left|_{\Sigma^-_{2\pi/3}}\right. = u^H(D\phi)\left|_{\Sigma^-_{2\pi/3}}\right.-u^H(D\phi)\left|_{\Sigma^-}\right.,$$
and then, using again the definition of $u^H$ and $D^H$
\begin{equation}
\label{PbD6bis}
d_{2\pi/3}\left|_{\Sigma^-_{2\pi/3}}\right.=D\phi\left|_{\Sigma^-_{2\pi/3}}\right.-D^H \circ D\phi\left|_{\Sigma^-_{2\pi/3}}\right.=0.
\end{equation}
A uniqueness argument for the Dirichlet problem \eqref{PbD4}, \eqref{PbD6}, \eqref{PbD6bis} yields $v_{2\pi/3}=v_0$ in $\Omega^H \cap (\Theta_{2\pi/3} \Omega^H).$
\\\\
With the same argument, we can construct three solutions of the Helmholtz equation  $v_0,~v_{2\pi/3},~v_{4\pi/3}$ respectively in the domains $\Omega^H,~\Theta_{2\pi/3} \Omega^H$ and $\Theta_{4\pi/3} \Omega^H$ and which coincide in the domains where they are jointly defined:
\begin{itemize}
	\item $v_0=v_{2\pi/3}$ in $\Omega^H \cap (\Theta_{2\pi/3} \Omega^H),$
	\item $v_{2\pi/3}=v_{4\pi/3}$ in $(\Theta_{2\pi/3} \Omega^H) \cap (\Theta_{4\pi/3} \Omega^H),$
	\item $v_{4\pi/3}=v_{0}$ in $(\Theta_{4\pi/3} \Omega^H) \cap \Omega^H.$
\end{itemize}
Thus we can construct a function $U^e \in H^1(\Omega^e)$ defined in $\Omega^e$ by
\[
U^e\left|_{\Theta_{k\pi/3}\Omega^e}\right.=v_{k\pi/3},\qquad k \in \{0,2,4\},
\]
so that $U^e$ is an $H^1$-function which satisfies
$$\Delta U^e + \rho U^e =0,\qquad \hbox{in }\Omega^e,$$
with homogeneous Dirichlet condition on $\Sigma^i$. From the uniqueness for the exterior problem we get $U^e=0$ in $\Omega^e$ and therefore $D\phi|_{\Sigma^i}=U^e|_{\Sigma^i}=0$.
This concludes the proof of Theorem~\ref{THsec3_3}.
\end{proof}

\subsection{Characterization of the half-space DtD operator}\label{sub:DH_def}

\noindent Let us recall the definition of $D^H:$
\begin{center}
\begin{tabular}{ccc}
$D^H: H^{1/2}(\Sigma^H)$ & $\longrightarrow$ & $H^{1/2}(\Sigma^H)$ \\[5pt]
$\psi$ & $\longmapsto$ & $ \left\{ \begin{array}{ll} D^H \psi\left|_{\Sigma^-}\right.\equiv u^H(\psi)\left|_{\Sigma^-_{2\pi/3}}\right. \\[6pt]  D^H \psi\left|_{\Sigma^0}\right.\equiv u^H(\psi)\left|_{\Sigma^{0}}\right. \\[5pt] D^H \psi\left|_{\Sigma^+}\right.\equiv u^H(\psi)\left|_{\Sigma^+_{2\pi/3}}\right.
\end{array}\right., $
\end{tabular}
\end{center}
where we have identified $\Sigma^-$ and $\Sigma^-_{2\pi/3},~\Sigma^+$ and $\Sigma^+_{2\pi/3}$ taking into account the directions shown in Figure~\ref{notations_SigmaH}.
\\\\ 
Thanks to the results of Section \ref{SPbdemiesp}, we can give a semi-analytic expression for the DtD operator $D^H$.
\begin{proposition}\label{prop:DH_expression}
	For any $\psi$ in $H^{1/2}(\Sigma^H)$ and for any $\xi$ in $(-\pi/L,\pi/L)$, we have
	\begin{eqnarray}\label{eq:DH_express}
	\cF_y(D^H\, \psi)(\cdot,\xi)&=&\frac{L}{2\pi}\int_{-\pi/L}^{\pi/L}\-e^{\imath \xi L}\left(\mathcal{D}_{k}^{\ell +} +\mathcal{D}_{k}^{r+}\cP_{k}\right)\left(\mathcal{I}-\cP_{k}e^{\imath\xi L}\right)^{-1}\,R^{QP}_k\widehat{\psi}_k\,dk \nonumber\\
	&+&  \frac{L}{2\pi}\int_{-\pi/L}^{\pi/L}\widehat{\psi}_k\left|_{\Sigma^0}\,dk\right.\\
	&+&\frac{L}{2\pi}\int_{-\pi/L}^{\pi/L}\-e^{-\imath (k+\xi) L}\left(\mathcal{D}_{k}^{\ell -} +\mathcal{D}_{k}^{r-}\cP_{k}\right)\left(\mathcal{I}-\cP_{k}e^{-\imath(k+\xi) L}\right)^{-1}\,R^{QP}_k\widehat{\psi}_k\,dk \nonumber,\nonumber
	\end{eqnarray}
	where $\widehat{\psi}_k=\cF_y\psi(\cdot,k)$, the local DtD operators $\mathcal{D}_{k}^{\ell \pm}$ and $\mathcal{D}_{k}^{r \pm}$ are given in Definition \ref{defDtD} and $\cP_{k}$ is the propagation operator defined in Section \ref{ssub:the_propagation_operator}.
\end{proposition}
\begin{remark}
Let us remark that the subscript $k$ of the operators in formula \eqref{eq:DH_express} means that they actually depend on the variable $k$. 
\end{remark}

\begin{proof}
	Let $k\in (-\pi/L,\pi/L)$, we express first $D^H\,\widehat{\psi}_k$ for $k-$quasiperiodic boundary data $\widehat{\psi}_k\in H^{1/2}_k(\Sigma^H)$. By definition of $D^H$, we are then interested in expressing $\widehat{u}^H_k(\widehat{\psi}_{k})\left|_{\Sigma^-_{2\pi/3}}\right.$, $\widehat{u}^H_k(\widehat{\psi}_{k})\left|_{\Sigma^0}\right.$ and $\widehat{u}^H_k(\widehat{\psi}_{k})\left|_{\Sigma^+_{2\pi/3}}\right.$. Using notation of Section \ref{sub:PH_kQPphi}, let us remark that 
	$$\Sigma^-_{2\pi/3}=\bigcup_{n=+\infty}^0 \Gamma_{n0}^{+} ,~~~\Sigma^+_{2\pi/3}=\bigcup_{n=0}^{+\infty} \Gamma_{n,-(n+1)}^{-} ,$$
taking into account the directions shown in Figures \ref{notations_SigmaH} and \ref{notationscell}.
\\\\Using Corollary \ref{cor:trace_uH}, we have
\begin{itemize}
	\item on $\Sigma^-_{2\pi/3}$, for all $n \in \N$,  
	$$ \widehat{u}_k^H(\widehat{\psi}_k)\left|_{\Gamma^{+} _{n0}}\right.=\left(\mathcal{D}_k^{\ell +}  \, \cP^n_k + \mathcal{D}_k^{r +}\,\cP_k^{n+1}\right) \,R^{QP}_k\widehat{\psi}_k$$
	and $D^H\,\widehat{\psi}_k\left|_{\Sigma^-_n}\right.=\widehat{u}_k^H(\widehat{\psi}_k)\left|_{\Gamma^{+} _{n0}}\right.$ where we have denoted $\Sigma^-_n = \Theta_{-2\pi/3}\Gamma^{\ell +}_{n0};$
	\item on $\Sigma^0$
	$$ D^H\,\widehat{\psi}_k\left|_{\Sigma^0}\right. = \widehat{u}_k^H(\widehat{\psi}_k)\left|_{\Sigma^0}\right. =\widehat{\psi}_k\left|_{\Sigma^0}\right.;$$
	\item on $\Sigma^+_{2\pi/3}$, for all $n \in \N$, 
	$$ \widehat{u}_k^H(\widehat{\psi}_k)\left|_{\Gamma^{-} _{n,-(n+1)}}\right.= e^{-\imath(n+1)kL}\left(\mathcal{D}_k^{\ell -} \, \cP^n_k + \mathcal{D}_k^{r -} \, \cP_k^{n+1}\right)\,R^{QP}_k\widehat{\psi}_k$$
	and $D^H\,\widehat{\psi}_k\left|_{\Sigma^+_n}\right.=\widehat{u}_k^H(\widehat{\psi}_k)\left|_{\Gamma^{-} _{n,-(n+1)}}\right.$ where $\Sigma^+_n = \Theta_{2\pi/3}\Gamma^{-} _{n,-(n+1)}.$
\end{itemize}
For any data $\psi\in H^{1/2}(\Sigma^H)$, we use the FB transform and formula \eqref{eq:extenduH} to obtain the expression "piece by piece" of $D^H\psi$ from $\widehat{\psi}_k =\cF_y\psi(\cdot,k)$ 
\begin{itemize}
	\item on $\dsp\Sigma^-=\cup_{n\in\N}\Sigma^-_n$, for all $n \in \N$,  
	$$ D^H\,{\psi}\left|_{\Sigma^-_n}\right.=\sqrt{\frac{L}{2\pi}}\int_{-\pi/L}^{\pi/L}\left(\mathcal{D}_k^{\ell +}  \, \cP^n_k + \mathcal{D}_k^{r +}\,\cP_k^{n+1}\right) \,R^{QP}_k\widehat{\psi}_k\,dk;$$
	\item on $\Sigma^0$
	$$ D^H\,{\psi}\left|_{\Sigma^0}\right. = \sqrt{\frac{L}{2\pi}}\int_{-\pi/L}^{\pi/L}\widehat{\psi}_k\left|_{\Sigma^0}\,dk\right.;$$
	\item on $\dsp\Sigma^+=\cup_{n\in\N}\Sigma^+_n$, for all $n \in \N$, 
	$$ D^H\,{\psi}\left|_{\Sigma^+_n}\right.=\sqrt{\frac{L}{2\pi}}\int_{-\pi/L}^{\pi/L} e^{-\imath(n+1)kL}\left(\mathcal{D}_k^{\ell -} \, \cP^n_k + \mathcal{D}_k^{r -} \, \cP_k^{n+1}\right)\,R^{QP}_k\widehat{\psi}_k\,dk.$$
\end{itemize}
We finally apply the FB transform in the $y-$ direction to $D^H \,\psi$:
\begin{eqnarray}
\cF_y(D^H \,\psi)(\cdot,\xi)=\sqrt{\frac{L}{2\pi}}\left[\sum_{n=\infty}^0 D^H\,{\psi}\left|_{\Sigma^-_n}\right.e^{\imath\xi(n+1)L}+D^H\,{\psi}\left|_{\Sigma^0}\right. +\sum_{n=0}^\infty D^H\,{\psi}\left|_{\Sigma^+_n}\right.e^{-\imath \xi(n+1)L}\right]. \nonumber
\end{eqnarray}
By inverting the integral over $(-\pi/L,\pi/L)$ and the sum over $n$, we are led to use the following formula:
\begin{equation}\label{eq:sum_cv}
	\sum_{n\in\N}\cP_k^n e^{\pm\imath(n+1)\zeta L} = e^{\pm\imath\zeta L}\left(\mathcal{I}-\cP_ke^{\pm\imath\zeta L}\right)^{-1}
\end{equation}
with $\zeta=k$ or $\zeta = k+\xi$. Note that this is possible as for every $k$, $\cP_k$ is compact with spectral radius strictly less than $1$. Actually, we could prove like in \cite{TheseSonia} that for $ \rho_b>0$ (defined in \eqref{HypDissipation}), the spectral radius $\rho(\cP_k)$ of $\cP_k$ is uniformly bounded in $k$, by a constant that is strictly less than $1$:
\[
	\exists C>0,\;\forall k\in\left(-\frac{\pi}{L},\frac{\pi}{L}\right),\quad \rho(\cP_k)\leq e^{-C \rho_b}.
\]
The property 
\[
	\lim_{n\rightarrow +\infty}\norm{\cP_k^n}^{1/n} = \rho(\cP_k)
\]
for the norm of $\mathcal{L}(L^2(\Sigma_{00}^\ell))$ (\cite{Weidmann:1980}) implies that for some $\alpha\in ]e^{-C \rho_b},1[ $, $j$ large enough we have for all $k$
\[
	\norm{\cP_k^j}\leq \alpha^j,
\]
which yields the absolute convergence of the series \eqref{eq:sum_cv}. Therefore, for each $\zeta$, $\mathcal{I}-\cP_ke^{\pm\imath\zeta L}$ is invertible and the sum \eqref{eq:sum_cv} converges uniformly in the norm of $\mathcal{L}(L^2(\Sigma_{00}^\ell))$. Exchanging the order of the integral and the sum is then possible.
\end{proof}

\begin{remark}\label{rem:FBT}
	Our choice of solving a family of $k-$quasiperiodic half-space problems instead of a family of $k-$quasiperiodic half waveguide problems simplifies the expression of $D^H$. Indeed, expressing $u^H(\phi)$ on $\Sigma^\pm_{2\pi/3}$ is much easier by doing so. 
\end{remark}

\subsection{Towards the integral equation} 
\label{sub:towards_the_integral_equation}

Let us give now more precisions about the resolution of the affine equation \eqref{PbD}. Since the operator $D^H$ is characterized via its FB transform, it makes sense to reformulate \eqref{PbD} using the FB transform.
\begin{corollary}
	For any $\phi\in H^{1/2}_{2\pi/3}(\Sigma^i)$, the function
	\[
		\widehat{\psi}_{2\pi/3} := \cF_y(D_{2\pi/3}\phi)\quad\in H^{1/2}_{QP}\left(\Sigma^H\times\left(-\frac{\pi}{L},\frac{\pi}{L}\right)\right)
	\]
	is the unique solution to the following problem
	\[
		\mbox{Find } \widehat{\psi} \in H^{1/2}_{QP}\left(\Sigma^H\times\left(-\frac{\pi}{L},\frac{\pi}{L}\right)\right), \mbox{ such that },
	\]
	\begin{equation}
	\label{PbD_FB}
	\begin{array}{|cl}
		(i)&\dsp \forall\xi\in\left(-\frac{\pi}{L},\frac{\pi}{L}\right),\quad\widehat{\psi}(\cdot,\xi)-\int_{-\pi/L}^{\pi/L}K^H(k,\xi)\,\widehat{\psi}(\cdot,k)\,dk = 0,\\[7pt]
		(ii)&\dsp \sqrt{\frac{L}{2\pi}}\int_{-\pi/L}^{\pi/L}\widehat{\psi}(\cdot,k)\left|_{\Sigma^0}\right.\,dk\, = \phi\left|_{\Sigma^0},\right.
\end{array}	\end{equation}
	where $K^H(k,\xi)\in \mathcal{L}\left(H^{1/2}_k(\Sigma^H),H^{1/2}_\xi(\Sigma^H)\right)$ is given by
	\begin{eqnarray*}
	K^H(k,\xi)&=&\frac{L}{2\pi}\,e^{\imath \xi L}\left(\mathcal{D}_{k}^{\ell +} +\mathcal{D}_{k}^{r+}\cP_{k}\right)\left(\mathcal{I}-\cP_{k}e^{\imath\xi L}\right)^{-1}\,R^{QP}_k \nonumber\\
	&+&  \frac{L}{2\pi}R^H\\
	&+&\frac{L}{2\pi}e^{-\imath (k+\xi) L}\left(\mathcal{D}_{k}^{\ell +} +\mathcal{D}_{k}^{r+}\cP_{k}\right)\left(\mathcal{I}-\cP_{k}e^{-\imath(k+\xi) L}\right)^{-1}\,R^{QP}_k\nonumber, 
	\end{eqnarray*}
	with $R^H$ the restriction operator from $\Sigma^H$ on $\Sigma^0$ defined in Section \ref{Sect_FACT}.
\end{corollary}
\noindent Relation \eqref{PbD_FB}-(i) is the FB transform of Equation \eqref{PbD} while relation \eqref{PbD_FB}-(ii) expresses in terms of the FB-variables the condition:
\[
	D_{2\pi/3}\phi\,\left|_{\Sigma^0}\right. = \phi\,\left|_{\Sigma^0}\right..
\]
From a practical point of view, it seems to us that it is easier to solve \eqref{PbD_FB} instead of \eqref{PbD}. The advantage is to replace the discretization of an infinite set by the discretization of a compact set. 

\section{Summary and conclusion} 
\label{sec:summary_and_conclusions}
In this paper, we proposed a method to solve scattering problems in infinite hexagonal periodic media containing local perturbations. We have computed the Dirichlet-to-Neumann map on the boundary of a bounded domain which respects the hexagonal geometry. By doing so, the initial problem is reduced to the solution of a boundray value problem set in a cell containing the defect. The computation of the DtN map $\Lambda$ is based on a factorization through a half-space DtN operator $\Lambda^H$ and a DtD operator $D_{2\pi/3}$ (see Theorem \ref{THsec3_2}). We sum up below the main steps to be followed in order to compute $\Lambda\phi$, for a given boundary data with hexagonal symmetry $\phi\in H^{1/2}_{2\pi/3}(\Sigma^i)$.
\begin{enumerate}
	\item Pre-processing steps
	\begin{enum_spe}
		\item[For all $k\in(-\pi/L,\pi/L)$]\begin{enumerate}
		\item Solve the cell problems \eqref{CellPb1}-\eqref{CellPb2}
			\item Compute the local DtN operators \\ (see Definition \eqref{defDtN})
			\item Compute the local DtD operators \\ (see Definition \eqref{defDtD})
		\item Solve the stationary Riccati equation \eqref{Riccati}
		\item Compute $\widehat{\Lambda}^H_k$ thanks to relation \eqref{eq:Lambda_H_chapeau}
	\end{enumerate}
\end{enum_spe}
	\item Solve integral equation \eqref{PbD_FB} using the steps (1)-(c) and (1)-(d) to obtain $
		\widehat{\psi}_{2\pi/3}=\mathcal{F}_y\Big(D_{2\pi/3}\phi\Big) $
	\item Using relation \eqref{eq:lambdaH}, 
	\[
		\Lambda\,\phi = \sqrt{\frac{L}{2\pi}}\int_{-\pi/L}^{\pi/L}\widehat{\Lambda}^H_k\widehat{\psi}_{2\pi/3}(\cdot,k)\,dk
	\]
\end{enumerate}
The pre-processing steps only involve the solution of elementary problems set on bounded domains, for each values of $k$. These steps can be easily parallelized as the problems for different values of $k$ are decoupled.  Solving the integral equation \eqref{PbD_FB} constitutes the main difficulty in implementing the algorithm. The discretization and numerical investigation of the problem will be considered in a forthcoming article.


\end{document}